\documentclass[a4paper,reqno,11pt]{amsart}

\usepackage{amsmath, amsfonts, amssymb, amsthm, amscd}
\usepackage{graphicx}
\usepackage{psfrag}
\usepackage{perpage}
\usepackage{url}
\usepackage{color}
\usepackage{mathrsfs}
\usepackage{mathabx}
\usepackage{scalerel}
\usepackage{enumerate}
\usepackage{comment}
\usepackage[normalem]{ulem}
\usepackage{enumitem}
\usepackage{tikz}\usepackage{caption,subcaption}
\usetikzlibrary{arrows,decorations.pathmorphing,backgrounds,positioning,fit,petri}
\usepackage{tikz-cd} % Note taking
\usepackage{subdepth}

\usepackage{dsfont} % nice indicator function symbol (blackboard 1)

\usepackage[utf8]{inputenc}
\usepackage[T1]{fontenc}
\usepackage{microtype}

\usepackage[a4paper,scale={0.72,0.74},marginratio={1:1},footskip=7mm,headsep=10mm]{geometry}

\usepackage{hyperref}

\allowdisplaybreaks[0] % disable align page breaks

\numberwithin{equation}{section}

\newtheorem{theorem}{Theorem}[section]
\newtheorem{lemma}[theorem]{Lemma}
\newtheorem{proposition}[theorem]{Proposition}

\theoremstyle{definition}

\newtheorem{remark}[theorem]{Remark}

\newcommand{\E}{\mathbb{E}}
\newcommand{\N}{\mathbb{N}}
\renewcommand{\P}{\mathbb{P}}

\newcommand{\R}{\mathbb{R}}
\newcommand{\Z}{\mathbb{Z}}
\newcommand{\T}{\mathbb{T}}

\def\bs{\boldsymbol}

\newcommand\bP{\ensuremath{\bs{\mathrm{P}}}}

\newcommand{\cC}{{\ensuremath{\mathcal C}} }
\newcommand{\cH}{{\ensuremath{\mathcal H}} }
\newcommand{\cT}{{\ensuremath{\mathcal T}} }

\newcommand{\be}{\begin{equation}}
	\newcommand{\ee}{\end{equation}}

\definecolor{darkorange}{RGB}{255, 100, 0}

\newcommand{\e}{{\rm e}}
\newcommand{\ind}{\mathds{1}}

\newcommand{\eps}{\varepsilon}
\renewcommand{\eps}{\varepsilon}
\renewcommand{\theta}{\vartheta}
\renewcommand{\rho}{\varrho}

\newcommand{\tmix}{\ensuremath{t_{\text{\normalfont mix}}}}

\makeatletter
\newcommand\xleftrightarrow[2][]{%
\ext@arrow 9999{\longleftrightarrowfill@}{#1}{#2}}
\newcommand\longleftrightarrowfill@{%
\arrowfill@\leftarrow\relbar\rightarrow}
\makeatother

\MakePerPage[2]{footnote} % restarts footnote counter at every new page

\begin{document}
\date{\today}

\title[Diameter of UST on Random Weighted Graphs]{Diameter of Uniform Spanning Trees on Random Weighted Graphs}

\author[L. Makowiec]{Luca Makowiec}
\address{Department of Mathematics\\
	National University of Singapore\\
	10 Lower Kent Ridge Road, 119076 Singapore
}
\email{makowiec@u.nus.edu}

\author[M. Salvi]{Michele Salvi}
\address{
	Department of Mathematics\\
	University of Rome Tor Vergata\\
	Via della Ricerca Scientifica 1, 00133 Rome, Italy
}
\email{salvi@mat.uniroma2.it}

\author[R. Sun]{Rongfeng Sun}
\address{Department of Mathematics\\
	National University of Singapore\\
	10 Lower Kent Ridge Road, 119076 Singapore
}
\email{matsr@nus.edu.sg}

\begin{abstract}
For any edge weight distribution, we consider the uniform spanning tree (UST) on finite graphs with i.i.d.\ random edge weights. We show that, for bounded degree expander graphs and finite boxes of $\Z^d$, the diameter of the UST is of order $n^{1/2+o(1)}$ with high probability, where $n$ is the number of vertices.
\end{abstract}


%%%%%%%%%%%%%%%%%%%%%%%%%%%%%%%%%%%%%%%%%%%%%%
%%%% Main text entry area:

\maketitle

\setcounter{tocdepth}{2}
\makeatletter
\def\l@subsection{\@tocline{2}{0pt}{2.5pc}{5pc}{}}
\makeatother

\section{Introduction}

\subsection{Background and Main Result}

Let $(G, w)$ be a connected weighted finite graph, where $G=(V, E)$ has vertex set $V$ and edge set $E$, and $w:=(w_e)_{e\in E}$ are the weights (or conductances) assigned to the edges, with $w_e>0$ for all $e\in E$. A {\em spanning tree} $T$ of $G$ is a cycle--free connected subgraph of $G$ with the same vertex set $V$. We identify $T$ with its own edge set and write $\T=\T(G)$ for the set of all spanning trees on $G$.
The {\em uniform spanning tree} (UST) on $(G, w)$ is then defined to be the random spanning tree $\cT$ on $G$ with probability distribution
\be\label{eq:UST}
\bP^{w}(\cT = T) = \frac{1}{Z} \prod_{e\in T} w_e \qquad \mbox{with }\ Z=Z(w):= \sum_{T\in \T} \prod_{e\in T} w_e\,.
\ee
%where $\T=\T(G)$ is the set of spanning trees on $G$, i.e., the edge set of cycle free subgraphs of $(V, E)$ with the same vertex set $V$.
The UST is a fundamental object in combinatorics and probability, which has interesting connections to electric networks, loop erased random walks, percolation, dimers and many other topics, see e.g.~\cite{Pem95, Bar16, LP16, Law18} for more background on the UST. We point out that most studies of the UST are on unweighted graphs where $w \equiv 1$.

\smallskip	

One fundamental question concerns the scaling limits of the UST on sequences of large finite graphs. To identify the limit, the first step is to identify the correct order of the diameter $\mathrm{diam}(\cT)$, that is, the maximal graph distance in the UST between any pair of vertices.  For unweighted {\em ``high-dimensional''}   graphs, such as the complete graph, finite tori in dimension $d \geq 5$, expanders and dense graphs, it has been shown in \cite{Sze83, Ald90, CHL12, PR05, MNS21, ANS22} that the diameter of the UST is typically of order $\sqrt{n}$, where $n$ is the number of vertices in the graph. In fact, it is believed that, seen as a random metric space equipped with the graph distance, the UST rescaled by $1/\sqrt{n}$ would converge in distribution to Aldous' {\em continuum random tree (CRT)} \cite{Ald91a, Ald91b, Ald93}. This was verified in the Gromov-Hausdorff topology for the complete graph in \cite{Ald91a, Ald91b, Ald93}, in finite-dimensional distribution for finite tori in dimension $d\geq 5$ and $d=4$ in \cite{PR05} and \cite{Sch09} respectively, and in the stronger Gromov-Hausdorff-Prohorov topology for finite tori in dimension $d\geq 5$ in \cite{ANS21} and for dense graphs in \cite{AS24}.

\smallskip

Our goal in the present work is to initiate the study of the UST on random weighted graphs with i.i.d.\ edge weights, which defines a disordered system similar to {\em random walks in random environment}, a topic that has been studied extensively, see e.g.~\cite{Szi04, Zei04}. The basic question is how the random environment affects the behaviour of the UST. More specifically, we study the diameter of the UST in a typical environment.
%More specifically, we study the diameter of the UST in a typical random environment. The basic question is how the random environment affects the behaviour, such as the diameter, of the UST. 
We point out that the techniques developed for unweighted graphs, such as in \cite{MNS21}, no longer apply when the edge weights are not uniformly bounded away from $0$ and $\infty$.
We will treat two separate cases for the underlying graph: expander graphs with bounded degrees (see Section \ref{S:expmixing} for the proper definition) and boxes in the $\Z^d$ lattice, with $d\geq 5$. We show that in both cases, regardless of the edge weight distribution, the diameter of the UST in a typical random environment is of order $\sqrt{n}$ modulo a factor of $(\log n)^c$. Our main result is the next theorem.

%\begin{theorem} \label{T:main}
%Let $G = (V,E)$ be a graph with $|V|=n$, which belongs to one of the following two cases:
%\begin{itemize}
%\item[(i)] $G$ is a $b$-expander graph for some $b>0$ and has maximum degree $\Delta<\infty$;
%\item[(ii)] $G$ is the box $[-L, L]^d\cap \Z^d$ of volume $n$, for some $d\geq 5$.
%\end{itemize}
%Let $(w_e)_{e\in E}$ be i.i.d.\ random edge weights with common distribution $\mu$ satisfying $\mu(0, \infty) = 1$. Denote probability and expectation w.r.t.\ $w$ by $\P$ and $\E$. Given $w$, let $\mathcal{T}$ be the UST on the weighted graph $(G, w)$ defined as in \eqref{eq:UST}, with its law denoted by $\bP^w$. Then there exist $c, \gamma>0$ depending only on $\mu$, and (for case (i)) $b$ and $\Delta$, or (for case (ii)) $d$,
%such that for all large $n$ and $\epsilon > n^{-\gamma}$,
%    \begin{equation*}
	%        \widehat \P \left( \frac{\sqrt{n}}{(\eps^{-1}\log n)^c} \leq \text{diam}(\cT) < (\eps^{-1} \log n)^c \sqrt{n}  \right) \geq 1-\eps\,,
	%    \end{equation*}
%    where $\widehat \P(\cdot)$ denotes the averaged law $\E\bP^w(\cdot)$.
%\end{theorem}

\begin{theorem} \label{T:main}
	Let $G = (V,E)$ be a graph with $|V|=n$, which is either of the following: %which belongs to one of the following two families:
	\begin{itemize} 
		\item[(i)] a $b$-expander graph for some $b>0$ with maximum degree $\Delta<\infty$; 
		\item[(ii)] the box $[-L, L]^d\cap \Z^d$ of volume $n$, for some $d\geq 5$.
	\end{itemize}
	Let $(w_e)_{e\in E}$ be i.i.d.\ random edge weights with common distribution $\mu$ satisfying $\mu(0, \infty) = 1$. Denote probability and expectation w.r.t.\ $w$ by $\P$ and $\E$. Given $w$, let $\mathcal{T}$ be the UST on the weighted graph $(G, w)$ defined as in \eqref{eq:UST}, with its law denoted by $\bP^w$. Then there exist constants $c_1, c_2 > 0$ (with $c_1 = c_1(\mu, b\, \Delta)$ for case (i), and $c_1 = c_1(\mu, d), c_2 = c_2(d)$ for case (ii)) %\ref{itm:boxes}) refers to the theorem instead... 
    and $\gamma>0$ such that for all $n \geq 2$ and $\epsilon > n^{-\gamma}$,
	\begin{equation}\label{eq:probbdd}
		\widehat \P \left( \frac{\sqrt{n}}{c_1 (\eps^{-1}\log n)^{c_2}} \leq \mathrm{diam}(\cT) \leq c_1 (\eps^{-1} \log n)^{c_2} \sqrt{n}  \right) \geq 1-\eps\,,
	\end{equation}
	where $\widehat \P(\cdot)$ denotes the averaged law $\E\bP^w(\cdot)$. %The constant $\gamma$ does not depend on the other parameters of the model in case (i) and only depends on $d$ in case (ii). The constant $c$ depends on $\mu$, $b$ and $\Delta$ in case (i), and on $\mu$ and $d$ in case (ii).
\end{theorem}

One may apply Theorem \ref{T:main} as follows.
\begin{remark}\label{R:seq}
	%One may apply Theorem \ref{T:main} as follows. 
	Consider fixed $\mu$ with $\mu(0,\infty)=1$, $b > 0$ and $\Delta < \infty$ and let $(G_n)_{n \in \N}$, with $|V_n| = n$, be a sequence of graphs such that each $G_n$ is a $b$-expander with maximum degree at most $\Delta$ and i.i.d.\ edge weights $w_n$ distributed according to $\mu$. Then, since the constants in Theorem \ref{T:main} are independent of $n$,  the diameter of $\cT_{(G_n, w_n)}$ is of order $\sqrt{n}$ up to polylogarithmic factors with high probability as $n\to\infty$. %not to be confused with for all n simultaneously it is of order sqrt n
\end{remark}
\begin{remark} \label{R:not_bounded}
	The conclusion in Remark \ref{R:seq} cannot hold in general if we drop the condition that graphs in the sequence $(G_n)_{n\in\N}$ have uniformly bounded maximal degrees, because the upper and lower bounds on $\mathrm{diam}(\cT)$ in \eqref{eq:probbdd} depend on the maximal degree of the graph. For example, in Section \ref{S:counterex} we show that the diameter of the UST on the complete graph is typically of order $n^{1/3+o(1)}$ if the law of the edge weights is very heavy-tailed. 
	%then the UST on the complete graph has typically a diameter of order $n^{1/3}$, 
	%In that case, the diameter behaves like that of the \textit{minimum spanning tree}. 
\end{remark}

For further studies, it will be interesting to investigate whether the factors of $(\log n)^c$ can be removed from the bounds on the diameter for any choice of $\mu$. If the answer is positive, a natural question is whether one can show convergence along a sequence $G_n=(V_n, E_n)$ (with $|V_n|=n$) to Aldous' continuum random tree under either the averaged law $\widehat\P$, or the quenched law $\bP^w$ for typical realisations of the random edge weights $w$.

\subsection{Proof Strategy} \label{S:outline}

For unweighted graphs $G=(V, E)$, general conditions have been formulated in \cite{MNS21} that imply the diameter of the UST on $G$ to be
typically of order $\sqrt{|V|}$. However, these conditions can not be applied in our setting with i.i.d.\ random edge weights when the support
of the edge weight distribution $\mu$ is not bounded away from $0$ and $\infty$.

Our novel idea is to first single out edges whose weights
(or conductances) fall inside an interval $[1/A, A]$ for some $A>0$, which form a bond percolation process on $G$ with parameter $p\in (0,1)$ that
can be chosen arbitrarily close to $1$ by choosing $A$ large. We then condition on the edge configuration of the UST $\cT$ on the remaining closed edges, i.e., the edges with weight outside of $[1/A, A]$. This conditioning essentially allows us to consider the UST on a modified graph with edge weights uniformly bounded away from 0 and $\infty$. More precisely, by the spatial Markov property (see Lemma \ref{L:SMarkov} below), the conditional law of $\cT$ on the (open) edges with weights inside $[1/A, A]$ is the same as the law of a UST $\cT'$ on a new graph $G'$ (possibly with  multiple edges) obtained from $G$, where each closed edge is contracted if it lies in $\cT$ and deleted if it does not lie in $\cT$. We will then show that for typical realisations of the random edge weights and uniformly w.r.t.\ the realisation of $\cT$ on the closed edges, the graph $G'$ satisfies the conditions of Theorem \ref{T:Nach} below, which is a strengthened version of \cite[Theorem 1.1]{MNS21} and implies bounds on the diameter of $\cT'$. Undoing the contractions will then give us the desired bounds on the diameter of $\cT$.

To verify the conditions in Theorem \ref{T:Nach} for $G'$ in the case of expander graphs, we will analyse the \textit{bottleneck ratio} (see \eqref{eq:PhiG}) of $G'$ and show that it matches the bottleneck ratio of the unweighted version of $G$ up to polylogarithmic factors. This in turn will give us strong enough bounds on the mixing time of the lazy random walk on $G'$ needed to apply Theorem \ref{T:Nach}. For finite boxes in $\Z^d$, we will need to go a step further and analyse the whole \textit{bottleneck profile} (see \eqref{eq:BottleProfile}) of $G'$ and again show that this is up to polylogarithmic factors the same as that of $G$. Analysing the bottleneck profile instead of the bottleneck ratio allows us to obtain sharper bounds on the random walk transition kernel, which are needed for the case of finite boxes in $\Z^d$.

There is hope that this approach can be generalised to arbitrary bounded degree graphs with good enough expansion properties. We refer to \eqref{eq:ConjIso} in Section \ref{S:GeneralCounter} for the technical condition that is required to extend this result to other graphs.

\subsection{Outline}
The rest of the paper is organised as follows. In Section \ref{S:prelim}, we first recall some background material on the UST and expander graphs. We then formulate Theorem \ref{T:Nach}, which gives a variant of the conditions in \cite{MNS21} to bound the diameter of a UST on weighted graphs. In Section \ref{S:expbound} we show that the graph $G'$ discussed above has good expansion properties, and then in Section \ref{S:pf} we 
verify the conditions of Theorem \ref{T:Nach} for $G'$ and deduce Theorem \ref{T:main}. Section \ref{S:Zd} treats the case of finite boxes in $\Z^d$ with $d\geq 5$. In Section \ref{S:GeneralCounter}, we give a
counter-example to Remark \ref{R:seq} where the graph degrees are not uniformly bounded, and we discuss possible extensions to other graphs.  Lastly, we sketch in Appendix \ref{S:app} the proof of Theorem \ref{T:Nach}.

% Step 1: Fix some $\delta > 0$ small and choose $A > 0$ large enough such that $\mu\{ (0, 1/A) \cup (A, \infty)) \leq \delta$.

% Step 2: Condition on $\delta$ closed edges $K$ of the spanning tree $\mathcal{T}$

% Step 3: Contract edges in $G_n \mid_{\mathcal{T}(K)} \cap K$ to give new network $(\bar{G}_n, \bar{w}^n)$

% Step 4: Apply Nachmias to this new graph

% Step 5: Decontract graph to give bound on whole thing

\section{Preliminaries} \label{S:prelim}

We recall the spatial Markov property of the UST on a weighted graph $(G, w)$ in Section \ref{S:SMarkov} and the definition of edge expansion for $(G, w)$ and its connection
to the mixing time of a lazy random walk on $(G, w)$ in Section \ref{S:expmixing}. Finally, we give conditions that ensure that the diameter of the UST on $(G, w)$ is of order $|V|^{\frac{1}{2}+o(1)}$ with high probability in Section \ref{S:diamcond}.

\subsection{Spatial Markov property of UST}\label{S:SMarkov}

Given a graph $G = (V,E)$, the contraction of an edge $e\in E$ is the graph $G / e$ obtained by removing $e$ and identifying the endpoints of $e$ as a single vertex. The deletion of $e$ is the graph, denoted by $G-\{e\}$, with vertex set $V$ and edge set $E\backslash\{e\}$.
For $A \subseteq E$, the graph $G / A$, resp.\ $G - A$, is defined as the repeated contraction, resp.\ deletion, of all edges in $A$, which can be shown to be independent of the order of contraction, resp.\ deletion.

Given a finite connected weighted graph $(G, w)$, we will let $\cT_{(G,w)}$ denote the UST
on $(G, w)$. The UST is known to satisfy the following spatial Markov property, see e.g.\ \cite[Sec.~2.2.1]{HN19}.

\begin{lemma}\label{L:SMarkov}
	Let $(G, w)$ be a finite connected weighted graph. Let $A, B\subset E$ be two disjoint sets of edges such that
	$\bP(A\subset \cT_{(G,w)}, B\cap \cT_{(G,w)}=\emptyset)>0$. Then for any set of edges $F\subset E$,
	$$
	\bP(\cT_{(G, w)} = F | A\subset \cT_{(G, w)}, B\cap \cT_{(G, w)}=\emptyset) = \bP(\cT_{((G-B)/A, w)} \cup A = F).
	$$
\end{lemma}
\noindent
Namely, conditioned on $\cT_{(G, w)}$ containing all edges in $A$ and none of the edges in $B$, the law of $\cT_{(G, w)}$ restricted to $E\backslash (A\cup B)$ is the same as that of $\cT_{((G-B)/A, w)}$, the UST on the weighted graph where we have deleted all edges in $B$ and then contracted all edges in $A$.

\subsection{Edge expansion and mixing time}\label{S:expmixing}

Let $G=(V, E)$ be a finite graph. The edge expansion of a set of vertices $S\subset V$ is defined as
\begin{equation*}
	h_G(S) := \frac{|E(S, S^c)|}{|S|},
\end{equation*}
where $E(S, S^c)$ denotes the edges between $S$ and $S^c:=V\backslash S$. The {\em isoperimetric constant} or the {\em Cheeger constant} of $G$ (see e.g.\ \cite{Moh89}) is then  defined by
\begin{equation}\label{eq:hG}
	h_G := \min\limits_{1 \leq |S| \leq \frac{|V|}{2}} h_G(S).
\end{equation}
Given $b>0$, $G$ is called a {\em $b$--expander} if $h_G \geq b$, which is equivalent to
\be \label{eq:hG2}
|E(S, S^c)| \geq b \min \{ |S|, |S^c|\} \qquad \mbox{for all } S\subset V.
\ee

Consider now weights $(w_e)_{e \in E}$ on the edges of $G$. To avoid periodicity issues, one
typically considers the discrete-time lazy random walk $X$ with one-step transition probability
\begin{equation}\label{eq:lazy}
	q(x,y) = \left\{
	\begin{aligned}
		\frac{1}{2}\qquad  & \qquad \mbox{ if } x=y, \\
		\frac{1}{2}\frac{w_{\{x, y\}}}{\sum_z w_{\{x, z\}}} & \qquad \mbox{ if } x\neq y,
	\end{aligned}
	\right.
\end{equation}
and $t$-steps transition probabilities $q_t(x,y)$. The stationary distribution $\pi$ of $X$ satisfies
\begin{equation*}
\pi(x) = \frac{\sum_{v\in V} w_{\{x, v\}}}{\sum_{u, v\in V} w_{\{u, v\}}}.
\end{equation*}

The notion of edge expansion and isoperimetric constant can be extended to the weighted graph $(G, w)$ as follows. For $S\subset V$, the {\em bottleneck ratio} (also called {\em conductance}) of $S$ is defined as
\begin{equation}\label{eq:PhiGS}
	\Phi_{(G, w)}(S) := \frac{\sum_{e \in E(S, S^c)} w_e}{2 \sum_{x \in S, y\in V} w_{\{x, y\}}},
\end{equation}
where $w_{\{x,y\}}=0$ if $\{x, y\}\notin E$. The following quantity, which we will call the {\em bottleneck ratio} of $(G, w)$, defines an analogue of the isoperimetric constant for weighted graphs:
\begin{equation}\label{eq:PhiG}
\Phi_{(G, w)} := \min\limits_{0 < \pi(S) \leq 1/2} \Phi_{(G, w)}(S), \quad \mbox{where} \ \ \pi(S) = \sum_{x \in S} \pi(x) = \frac{\sum_{x \in S, y\in V} w_{\{x, y\}}}{\sum_{x, y\in V} w_{\{x, y\}}}.
\end{equation}
We remark that when $G$ is $d$-regular with constant weights, then the definitions in \eqref{eq:hG} and \eqref{eq:PhiG} differ up to a multiplicative constant in $[d, 2d]$. Furthermore, we note that given $b>0$, $\Phi_{(G, w)} \geq b$ is equivalent to
\begin{equation}\label{eq:PhiG2}
\sum_{e \in E(S, S^c)} w_e \geq 2 b \min\Big\{ \sum_{x \in S, y\in V} w_{\{x, y\}}, \sum_{x \in S^c, y\in V} w_{\{x, y\}} \Big\} \quad \mbox{for all } S\subset V,
\end{equation}
and we may rewrite the bottleneck ratio as 
\begin{equation}\label{eq:PhiGS2}
	\Phi_{(G, w)}(S)= \frac{\sum_{x\in S, y\in S^c} \pi(x) q(x,y)}{ \pi(S)},
\end{equation}
where we require the factor of $2$ in \eqref{eq:PhiGS} as the random walk is lazy. 

The (uniform) mixing time of the lazy random walk $X$ on $(G, w)$ is defined as
\begin{equation}\label{eq:tmix}
	\tmix(G, w) := \min \left\{ t \geq 0 : \max\limits_{u,v \in V} \left| \frac{q_t(u,v)}{\pi(v)} - 1\right| \leq \frac{1}{2}\right\}.
\end{equation}
We have the following relations between $\tmix(G, w)$ and $\Phi_{(G, w)}$.
\begin{theorem}[Cheeger Bound] \label{T:Cheeger}
	The mixing time $\tmix(G, w)$ of the lazy walk on $(G, w)$ and the bottleneck ratio $\Phi_{(G, w)}$ satisfy
	\begin{equation*}
		\frac{1}{4 \Phi_{(G, w)}} \leq \tmix(G, w) \leq  \frac{2 \log(2/\pi_{\min})}{\Phi_{(G, w)}^2},
	\end{equation*}
	where $\pi_{\min} = \min_{v \in V} \pi(v)$.
\end{theorem}
This result goes back to \cite{SJ89}, we also refer to Chapters 7, 12, and 13 of \cite{LP17}.

\subsection{Diameter of the UST}\label{S:diamcond}

In \cite{MNS21}, the authors considered finite unweighted graphs $G=(V, E)$ with $|V|=n$. Under three conditions on $G$, they showed that the UST on $G$ has diameter of order $\sqrt{n}$ with high probability. We state here the analogue of their conditions for a weighted graph $(G, w)$ and remark that the main difference is in \eqref{eq:balanced}, which coincides with their original condition when $w \equiv 1$, in which case \eqref{eq:balanced} says that the ratio of maximum to minimum degree is bounded. We say that $(G,w)$ is balanced, mixing and escaping with parameters respectively $D, \alpha, \theta > 0$ if the following are satisfied:
\begin{enumerate}
	\item $(G, w)$ is balanced if
	\begin{equation} \label{eq:balanced}
		\frac{\max_{u\in V} \pi(u)}{\min_{u\in V} \pi(u)}= \frac{\max_{u \in V} \sum_{v} w_{\{u,v\}}}{\min_{u \in V} \sum_{v} w_{\{u,v\}}} \leq D;
	\end{equation}
	\item $(G, w)$ is mixing if
	\begin{equation} \label{eq:mixing}
		\tmix(G, w) \leq n^{\frac{1}{2} - \alpha}\,;
	\end{equation}
	\item $(G, w)$ is escaping if
	\begin{equation} \label{eq:escaping}
		\sum_{t=0}^{\tmix} (t+1) \sup_{v \in V}q_t(v,v) \leq \theta\,.
	\end{equation}
\end{enumerate}
In \cite{MNS21}, the bound on the diameter of the UST on an unweighted graph $G$ was formulated in terms of the constants $D, \alpha, \theta$ which do not depend on $n$. We formulate here an extension that includes weighted graphs and also allows $D$ and $\theta$ to depend on $n$.

\begin{theorem}[Extension of Theorem 1.1 in \cite{MNS21}] \label{T:Nach}
	For any $\alpha>0$, there exist $C, k, \gamma > 0$ such that if $(G,w)$ satisfies condition \eqref{eq:mixing} with $\alpha$ and conditions \eqref{eq:balanced} and \eqref{eq:escaping} for some $D=D(n)$ and $\theta=\theta(n)$ with $D,\theta \leq n^{\gamma}$, then for any $\epsilon > n^{-\gamma}$,
	\begin{equation}\label{eq:Nach1}
	\bP^{w}\big( (CD \theta \epsilon^{-1})^{-k} \sqrt{n} \leq \mathrm{diam}(\cT_{(G, w)}) \leq (CD \theta \epsilon^{-1})^{k} \sqrt{n}\big) \geq 1 - \epsilon.
	\end{equation}
\end{theorem}
We will apply Theorem \ref{T:Nach} to $(G, w)$ with $D, \theta \leq (\log n)^c$ for some $c>0$. The proof of Theorem \ref{T:Nach} will be sketched
in Appendix \ref{S:app}.

\section{Edge Expansion Bounds} \label{S:expbound}

For a weighted graph $(G, w)$ with arbitrary edge weights $w$, the constants $D$ and $\theta$ in \eqref{eq:balanced} and \eqref{eq:escaping} could be so large that the lower and upper bounds on the diameter in Theorem \ref{T:Nach} become too far apart to be meaningful. This happens in particular when the edge weights are i.i.d.\ random variables with a very heavy-tailed
distribution. In this case, there could be an edge whose weight dominates that of all other adjacent edges and the associated random walk would get stuck on that edge for a long time. As outlined in Section \ref{S:outline}, we circumvent this problem by conditioning on the UST $\cT$ restricted to edges whose weights lie outside the interval $[1/A, A]$, which are the closed edges in a percolation process. The conditional distribution of $\cT$ on the open edges is then a UST $\cT'$ on a new graph $(G',w')$ (with possibly multiple edges) where closed edges that lie in $\cT$ have been contracted while closed edges not in $\cT$ have been deleted.
The goal of this section is to give a lower bound on the bottleneck ratio for $(G', w')$ that is uniform over the realisation of $\cT$ on the closed edges
(see Prop.~\ref{P:Cheeger}) and uniform over typical realizations of the edge weights $w$ (that is, $w$ that satisfy the conditions in \eqref{eq:B1234}). Thanks to the relation between the bottleneck ratio and the mixing time in Theorem \ref{T:Cheeger}, this will guarantee that the conditions of Theorem \ref{T:Nach} for $(G',w')$ are fulfilled.

\smallskip

We notice that $(G',w')$ consists only of edges with weights in $[1/A, A]$, so that controlling the isoperimetric constant and the maximum degree in $G'$ is sufficient to give good lower bounds for the bottleneck ratio of $(G',w')$.
%Since $G'$ only consists of edges with weights in $[1/A, A]$, to lower bound the bottleneck ratio for $(G', w)$, it suffices to lower bound the isoperimetric constant for $G'$ and upper bound the maximum degree in $G'$.
We accomplish this by comparing $G'$ with $\mathcal C_1$, the largest connected component of open edges in $G$ (i.e., edges with weights in $[1/A, A]$). It is known that for $A$ large enough, the isoperimetric constant of $\mathcal C_1$ is at least $1/\log |V|$ (see Lemma \ref{L:C1}). The graph $G'$ is obtained from $\mathcal C_1$ by contracting some closed edges and attaching the vertices outside $\mathcal C_1$. A crucial observation (see Lemma \ref{L:RMC1}) is that $\mathcal C_1$ disconnects the remaining vertices of $G$ into components of size at most $\log |V|$, which implies that when we attach these components to $\mathcal C_1$ to obtain $G'$, the isoperimetric constant only changes by a factor of $(\log |V|)^c$.

\smallskip

In Section \ref{S:clustersize}, we state three elementary bounds on percolation cluster sizes. In Section \ref{S:expcluster}, we recall a bound on the isoperimetric constant of the largest percolation cluster $\mathcal C_1$. In Section \ref{S:conditioning}, we bound the bottleneck ratio for the weighted graph $G'$ described above.

\subsection{Bounds on Percolation Clusters Size}\label{S:clustersize}

Given a finite graph $G=(V, E)$ and a percolation parameter $p\in (0, 1)$, we can perform bond percolation on $G$ by independently keeping each edge with probability $p$ and deleting otherwise. Kept edges are also called open, while deleted ones are called closed. In this way, the graph is broken into multiple connected components (or clusters), which are regarded as subgraphs of $G$. For $\ell\in \N$, let $\mathcal C_\ell = C_{\ell}(p)$ denote the $\ell$-th largest open cluster (ties are broken arbitrarily) and let $|\mathcal C_\ell|$ denote the number of vertices in $\mathcal C_\ell$.

We collect here three bounds on the sizes of percolation clusters. The first bound states that for a $b$--expander graph $G=(V, E)$ (cfr.~Section \ref{S:expmixing}), if $p$ is close enough to $1$, then the size of the largest cluster $\mathcal C_1(p)$ is at least $\zeta n$, where $\zeta\in (0, 1)$ can
be made arbitrarily close to $1$ by choosing $p$ close to $1$. In this case, $\mathcal C_1$ is also called the {\em giant component}.

\begin{lemma} \label{L:Giant}
	Let $b>0$. Then for all $\zeta \in (0,1)$, there exists $p_\zeta \in (0,1)$ depending only on $b$ such that for all $p \in [p_\zeta, 1]$ and for all
	$b$-expander graphs $G = (V, E)$ with $|V|=n$,
	\begin{equation*}
		\P( |\mathcal C_1(p)| \geq \zeta n) \geq 1 - \e^{-n}.
	\end{equation*}
\end{lemma}
\begin{proof}
	Given the percolation configuration on $G$ with parameter $p$, and for $A\subset V$, let $E_p(A, A^c)$ denote the set of open edges connecting $A$ and $A^c$. For $\zeta \in (0,1)$, let us consider the event
	\begin{equation*}
		\cH(\zeta)
		:= \{ \exists \, A \subset V \text{ s.t.\ } |A|\geq \frac{\zeta n} {2}, \,\, |A^c| \geq \frac{(1-\zeta)n}{2}  \,\text{ and }\, E_p(A, A^c) = \emptyset\}\,.
	\end{equation*}
	Note that if $|\mathcal C_1(p)| < \zeta n$, then the event $\cH(\zeta)$ must hold. Indeed, since the clusters $(\mathcal C_\ell(p))_{\ell\geq 1}$ are decreasing in size,  there is some $k \in \N$ such that $|\bigcup_{\ell = 1}^k \mathcal C_\ell(p)| \in [\zeta n/2, \zeta n]$. Choosing $A$ to be the vertex set of
	$\bigcup_{\ell = 1}^k \mathcal C_\ell(p)$ then establishes the event $\cH(\zeta)$.
	
	Since $G$ is a $b$-expander, when $\cH(\zeta)$ holds, there must exist $A \subset V$ with 
	\begin{equation*}
		|E(A, A^c)| \geq b \min \{ \zeta, 1- \zeta\} \frac{n}{2} \quad \text{and} \quad E_p(A, A^c) = \emptyset.
	\end{equation*}
	The probability that all edges between such $A$ and $A^c$ are closed in the percolation configuration is at most $(1-p)^{b \min \{ \zeta, 1- \zeta\} \frac{n}{2}}$. Thus a union bound over all $A\subset V$ gives
	\begin{equation*}
		\P( \mathcal{H}(\zeta))
		\leq 2^n (1-p)^{b \min \{ \zeta, 1- \zeta\} \frac{n}{2}}
		\leq \e^{-n}
	\end{equation*}
	provided $p \geq p_\zeta := 1 - \exp \big( -\frac{2(1+\log 2)}{b\min\{\zeta, 1-\zeta\}} \big)$.
\end{proof}

The second bound of this section states that in a bounded degree graph with $n$ vertices, for a sufficiently small percolation parameter $p'$, the largest open cluster has size at most $\log n$. Equivalently, for $p$ close to $1$, the largest cluster formed by closed edges are of size at most $\log n$.

\begin{lemma} \label{L:scluster}
	For any $\Delta \in \N$ and $\eta>0$, there exist $C>0$ and $p'_\eta \in (0,1)$ such that for any $G = (V, E)$ with $|V| = n$ and maximum degree $\Delta$, and for all $p' \in [0, p'_\eta]$,
	\begin{equation*}
		\P( |\mathcal C_1(p')| \geq \log n) \leq \frac{C}{n^\eta}\,.
	\end{equation*}
\end{lemma}
\begin{proof}
	For $r\in \N$, let $\mathcal{G}_r$ denote the set of all possible connected subgraphs of $G$ with $r$ vertices, each of which contains at least $r-1$ edges. For a graph with $n$ vertices and maximal degree $\Delta$, it is known that (see e.g.\ \cite[Proof of Lemma 2.2]{ABS04}) $|\mathcal{G}_r| \leq \frac{n}{r}(\Delta \e)^r$. A union bound
	over all connected subgraphs with at least $\log n$ vertices then gives
	\begin{equation*}
		\P(|\mathcal C_1(p)| \geq \log n)
		\leq  \sum_{r = \log n}^{n} \frac{n}{r} (\Delta \e)^r p^{r-1}
		\leq \frac{n}{p \log n}  \cdot
		\frac{ (\Delta \e p)^{ \log n}}{1-\Delta \e p} \leq \frac{C'\Delta}{n^\eta},
	\end{equation*}
	provided $p\leq p_\eta := c/\Delta$ for some $c>0$ sufficiently small.
\end{proof}

The third and last bound controls the components' sizes after removing the giant component $\mathcal C_1$ from $G$. Namely, if the percolation parameter $p$ is close enough to $1$, then after removing all vertices in $\mathcal C_1(p)$ and their incident edges from $G$ (denote the resulting graph by $G\backslash V(\mathcal C_1)$), the connected components of $G\backslash V(\mathcal C_1)$ are typically all of size $\log n$ or less.
\begin{lemma} \label{L:RMC1}
	Let $G = (V, E)$ be a $b$-expander graph with $b>0$, $|V|=n$, and maximum degree $\Delta< \infty$.  For all $\eta >0$, there exist $C>0$ and $p_1 \in (0,1)$ depending only on $\eta, b$ and $\Delta$, such that for any percolation parameter $p \in [p_1, 1]$, the graph $G\backslash V(\mathcal C_1)$ satisfies
	\begin{equation*}
		\P\big( G\backslash V(\mathcal C_1) \text{ contains a connected component } H \text{ with } |H| \geq \log n\big) \leq  \frac{C}{n^{\eta}}.
	\end{equation*}
\end{lemma}

\begin{proof} First observe that by the definition of $G\backslash V(\mathcal C_1)$, for any connected component $H$ of $G\backslash V(\mathcal C_1)$ the external vertex boundary of $H$ in $G$ must be fully contained in $V(\mathcal C_1)$ and thus $E(V(H), V\backslash V(H)) = E(V(H), V(\mathcal C_1))$.
	By choosing $p$ close enough to $1$, Lemma \ref{L:Giant} ensures that $|\mathcal C_1(p)|\geq n/2$ with high probability. Restricted to this event, if there is a connected component $H$ of $G\backslash V(\mathcal C_1)$ with $|H|\geq \log n$, then we have $|H|\leq n/2$. By the expander property of $G$, the number of edges between $V(H)$, the vertex set of $H$, and its complement satisfies
	\begin{equation*}
		|E(V(H), V\backslash V(H))| = |E(V(H), V(\mathcal C_1))|\geq b |H|\geq b\log n.
	\end{equation*}
	By the definition of $\mathcal C_1(p)$, the edges in $E(V(H), V(\mathcal C_1))$ must all be closed in the percolation configuration. This event has probability at most $(1-p)^{b |H|}$.
	
	Recall the definition of $\mathcal{G}_r$ from the proof of Lemma \ref{L:scluster} and the inequality $|\mathcal{G}_r| \leq \frac{n}{r}(\Delta \e)^r$.
	%For $r\in \N$, let $\mathcal{G}_r$ denote the set of all possible connected subgraphs of $G$ with $r$ vertices. For a graph with $n$ vertices and maximal degree $\Delta$, it is known that (see e.g.\ \cite[Proof of Lemma 2.2]{ABS04}) $|\mathcal{G}_r| \leq \frac{n}{r}(\Delta e)^r$.
	A union bound
	over all connected subgraphs $H$ with $\log n\leq |H| \leq n/2$ then gives
	\begin{align*}
		&\P(G\backslash V(\mathcal C_1) \text{ contains a connected component } H \text{ with } |H| \geq \log n) \\
		\leq\ & \P\Big(|\mathcal C_1(p)| < \frac{n}{2}\Big)  + \sum_{r = \log n}^{n/2} \frac{n}{r} (\Delta \e)^r (1-p)^{br}
		\leq \e^{-n} +
		\frac{2 n}{\log n} (\Delta \e (1-p)^b)^{ \log n} \leq \frac{C}{n^\eta},
	\end{align*}
	which holds if we choose $p$ close enough to $1$ so that Lemma \ref{L:Giant} holds for $\zeta=1/2$, and $\Delta \e (1-p)^b$ is sufficiently small.
\end{proof}

\subsection{Edge Expansion for Giant Component}\label{S:expcluster}

We recall the following fact from \cite{ABS04}: consider an edge percolation procedure on a $b$--expander graph with bounded degree. If the percolation parameter $p$ is sufficiently close to $1$, then the giant component $\mathcal C_1(p)$ has an isoperimetric constant that is at least $1/\log |V|$  with high probability.

\begin{lemma}[Proposition 5.1 in \cite{ABS04}] \label{L:C1}
	Let $G = (V, E)$ be a $b$-expander graph with $b>0$, $|V|=n$, and maximum degree $\Delta< \infty$. For all $\eta > 0$, there exists $C>0$ such that for any percolation parameter $p \in [p_2, 1]$, where
	\begin{equation*}
		p_2 = 1 - \frac{1}{4} \e^{-\frac{2}{b}(\eta + 2) } \Delta^{-2/b},
	\end{equation*}
	the isoperimetric constant of the giant component $\mathcal C_1(p)$ satisfies
	\begin{equation}\label{eq:hC1tail}
		\P\Big( h_{\mathcal C_1(p)} \leq \frac{1}{\log n}\Big) \leq \frac{C}{n^\eta}.
	\end{equation}
\end{lemma}\noindent
The proof is the same as in \cite{ABS04}, except that we keep track of the dependency on $\eta$.

\subsection{Conditioning on High and Low Weight Edges} \label{S:conditioning}

We now consider $G=(V, E)$ with $|V|=n$ and i.i.d.\ random edge weights $w=(w_e)_{e\in E}$ with common distribution $\mu$. For some large $A>0$
to be chosen later, we call $e\in E$ open if $w_e\in [1/A, A]$ and call $e$ closed otherwise, which defines a bond percolation process on $G$ with
percolation parameter $p=p(A)=\mu([1/A, A])$. Recall that for
$\ell\in \N$, the set $\mathcal C_\ell = \mathcal C_\ell(p)$ denotes $\ell$-th largest open cluster.  Let $K\subset E$ denote the random set of closed edges. As outlined in Section \ref{S:outline} and at the beginning of Section \ref{S:expbound}, given $w$, we will condition the configuration of the UST $\cT_{(G, w)}$ on $K$,
i.e., condition on $\cT_{(G, w)}(K):= \cT_{(G, w)} \bigcap K$. To simplify notation, we will omit $(G, w)$ from the subscripts and just write $\cT$ and $\cT(K)$.

\begin{figure}[ht]
	\centering
	\includegraphics[width=16cm]{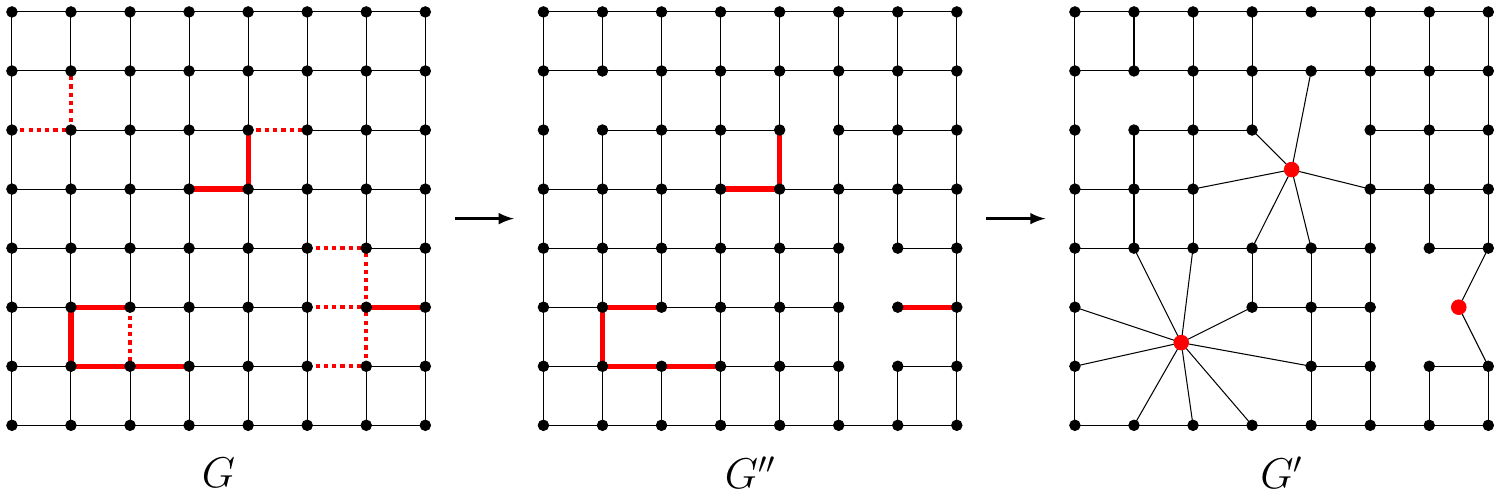}
	\caption{We first perform percolation on $G$ by closing edges with weight outside $[1/A, A]$ (drawn in red). We then condition on the realisation of the uniform spanning tree $\cT$ on the set of closed (red) edges, with edges in $\cT$ drawn in thick red lines, and edges not in $\cT$ drawn in dotted red lines. Deleting the dotted red edges leads to the graph $G''$, and further contracting the solid red edges leads to the graph $G'$.  
	}
	%Consider the graph $G$ on the left and perform percolation on $G$ by closing (indicated in red) edges with weight outside $[1/A, A]$ and call this set $K$. For any realization of $\cT$ on $K$ we can contract edges (drawn thick) in $\cT(K)$ and remove edges (drawn dotted) in $K \backslash \cT(K)$ such that we obtain a new graph on the right denoted by $G'$.}
\label{fig:ConDel}
\end{figure}

Given $\cT(K)$, we define a new graph (which may have multiple edges)
\be\label{eq:T'}
G' =G'(K, \cT(K)) := \Big(G-\big[ K  \setminus \cT(K)\big] \Big)/\cT(K),
\ee
where edges in $K\backslash \cT(K)$ are deleted and edges in $\cT(K)$ are contracted (cf.~Section \ref{S:SMarkov}), see Figure \ref{fig:ConDel}.  The edge set of $G'$ is exactly
the set of open edges $E\backslash K$, and we assign to each edge the original corresponding weight in $w$ and call the collection of their weights $w'$.
%, to which we assign the original weights $w$.
The goal of this section is to give a lower bound on the bottleneck ratio $\Phi_{(G', w')}$ for $(G', w')$ that is uniform both over the configuration of $\cT(K)$ and over all $w$ in the high probability event ${\mathcal B}:=\bigcap_{i=1}^4 B_i$. This is the event given by the intersection of
\be\label{eq:B1234}
\begin{aligned}
B_1 &:= \{|V(\mathcal C_1)| \geq 3n/4\}, \\
B_2 &:= \{ h_{\mathcal C_1} \geq 1/\log n \}, \\
B_3 &:= \{ \text{all connected components of } G \backslash V(\mathcal C_1) \text{ have size at most} \, \log n\}, \\
B_4 & := \{ \text{the closed edges $K$ form clusters of size at most} \, \log n\},
\end{aligned}
\ee
where we recall that $\mathcal C_{1}=\mathcal C_{1}(p)$ denotes the largest open cluster in the percolation process on $G$ with respect to the edge weights $w$.
%The high probability events for $w$ are defined as follows, where recall that $C_{\ell}$ denotes the $\ell$-th largest open cluster in the percolation process coupled to the random edge weights $w$:
Lemmas \ref{L:Giant}, \ref{L:scluster}, \ref{L:RMC1} and \ref{L:C1} imply that for any $\eta>0$,
\be\label{eq:Bprob}
\P(B_1 \cap B_2 \cap B_3\cap B_4) \geq 1 - \frac{C}{n^\eta},
\ee
provided that $A$ is chosen large enough so that $p=p(A)=\mu([1/A, A])$ is close enough to $1$.

\begin{proposition} \label{P:Cheeger}
Let $G = (V, E)$ be a $b$-expander with $b>0$, $|V|=n$, and maximum degree $\Delta< \infty$. Let $w=(w_e)_{e\in E}$ be i.i.d.\ random
edge weights with common distribution $\mu$.  Given any $\eta>0$, let $A$ be chosen such that \eqref{eq:Bprob} holds.
Let $G'=G'(K, \cT(K))$ be defined as in \eqref{eq:T'} with parameter $p(A)$. Then for all $w \in \bigcap_{i=1}^4 B_i$ and for $\bP^w$-a.e.\ $\cT(K)$,
\begin{equation*}
	\Phi_{(G', w')} \geq \frac{1}{16 A^2 \Delta^2 (\log n)^3 }.
\end{equation*}
\end{proposition}

Note that although all edges in $G'$ are assigned weights in $[1/A, A]$, which makes the bottleneck ratio $\Phi_{(G', w')}$ comparable to the isoperimetric
constant $h_{G'}$ of the unweighted graph $G'$, the degree of vertices in $G'$ can be arbitrarily large due to the contraction of edges. Therefore, to prove Prop.~\ref{P:Cheeger}, we first lower bound the isoperimetric constant $h_{G''}$ of the graph
\be\label{eq:G''}
G'' = G''(K, \cT(K)) :=G -  \big( K\backslash \cT(K) \big),
\ee
i.e., the graph obtained by deleting edges in $K\backslash\cT(K)$, but  without contracting edges in $\cT(K)$, see Figure \ref{fig:ConDel}. We then show that contraction will increase the maximal degree by a factor of $\log n$, which decreases the isoperimetric constant by a factor of at most $(\log n)^c$.

We have the following lower bound on $h_{G''}$.

\begin{lemma} \label{L:isop}
Assume the setting as in Proposition~\ref{P:Cheeger}. For all $w \in \bigcap_{i=1}^4 B_i$ and $\bP^w$-a.e.\ $\cT(K)$, we have
\begin{equation*}
	h_{G''}  \geq \frac{1}{8 \Delta (\log n)^2}.
\end{equation*}
\end{lemma}

\begin{proof}
Fix an arbitrary realisation of $\cT(K)$ under the law $\bP^w$, which determines the graph $G''$.
Let $S \subseteq V$ be any vertex set with $1\leq |S|\leq n/2$, which can be decomposed into $S_1 := S \cap V(\mathcal C_1)$ and $S_2 = S \backslash S_1$. Since the event $B_3$ occurs, $G \backslash V(\mathcal C_1)$ consists of disconnected components of size at most $\log n$. Let $L_1, \ldots, L_M$ denote the
components of $G \backslash V(\mathcal C_1)$ that contain some vertex in $S_2$. Then $|S_2|\leq M \log n$.

Note that $E(S, S^c)$ in particular contains all edges in $\mathcal C_1$ that connect $S_1$ to $V(\mathcal C_1)\backslash S_1$, which we denote here by $E_{\mathcal C_1}(S_1, S_1^c)$. Since the event $B_2$ occurs, we can use the isoperimetric constant of $\mathcal C_1$ to obtain
\be\label{eq:EC1S1}
\begin{aligned}
	|E_{\mathcal C_1}(S_1, S_1^c)| & \geq \frac{1}{\log n} \big( \ind_{\{|S_1|\leq |V(\mathcal C_1)|/2\}} |S_1| + \ind_{\{|V(\mathcal C_1)|/2< |S_1| \leq n/2 \}} |V(\mathcal C_1)\backslash S_1|\big) \\
	& \geq \frac{|S_1|}{2\log n},
\end{aligned}
\ee
where in the second line we used that, on the event $B_1$, the giant component satisfies $|V(\mathcal C_1)| \geq 3n/4$.

To bound $E(S, S^c)$, we distinguish between two cases. For the case $|S_1| \geq \frac{M}{2\Delta}$, we have
\begin{equation*}
	\frac{|E(S, S^c)|}{|S|} \geq \frac{|E_{\mathcal C_1}(S_1, S_1^c)|}{|S_1|+ M \log n} \geq \frac{\frac{1}{2\log n}}{1+ \frac{M \log n}{|S_1|}} \geq \frac{1}{2\log n(1+2\Delta\log n)} \geq \frac{1}{8\Delta (\log n)^2},
\end{equation*}
which satisfies the desired bound on $h_{G''}$.

We now consider the case $|S_1| < \frac{M}{2\Delta}$. Note that regardless of the realisation of $\cT(K)$,
$G''$ must remain a connected graph because $\cT$ connects all vertices in $V$. Therefore, for each $L_i$,
$1\leq i\leq M$, there is at least one edge in $G''$ connecting $L_i$ to $\mathcal C_1$. Since $G''$ has maximal degree $\Delta$, each vertex in $V(\mathcal C_1)$ can
connect to at most $\Delta$ different $L_i$'s. It follows that $S_1$ can connect to at most $\Delta |S_1|< M/2$ different $L_i$'s. Therefore, at least $M/2$ components among $L_1, \ldots, L_M$ are connected to some vertex in $V(\mathcal C_1)\backslash S_1$, and hence $E(S, S^c) \geq M/2$. It follows that
\begin{equation*}
	\frac{|E(S, S^c)|}{|S|} \geq \frac{M/2}{|S_1|+ M \log n} \geq  \frac{M/2}{\frac{M}{2\Delta}+ M \log n} \geq \frac{1}{1+2 \log n},
\end{equation*}
which also satisfies the desired bound on $h_{G''}$.

Since the above bounds hold for all $S\subset V$ with $1\leq |S|\leq n/2$ and are uniform in the realisation of $\cT(K)$, Lemma \ref{L:isop} follows.
\end{proof}

We are now ready to prove Prop.~\ref{P:Cheeger}.

\begin{proof}[Proof of Proposition \ref{P:Cheeger}]
Fix an arbitrary realisation of $\cT(K)$ under the law $\bP^w$, which determines the graphs $G''$ and $G'$. Recall that $G'$ is constructed from $G''$ by contracting all edges in $\cT(K)$, which is equivalent to contracting each connected component in the forest $\cT(K)$ into a single vertex. On the event $B_4$, the connected components formed by closed edges have size at most $\log n$. Since $G''$ has maximum degree $\Delta$, after
contraction, $G'$ has maximum degree at most $\Delta \log n$.

Let $S\subset V(G')$, and let $\tilde S\subset V(G'')$ be the pre-image of $S$ under the contraction that generates $G'$ from $G''$.
On the event $\bigcap_{i=1}^4 B_i$, it holds that $h_{G''}\geq 1/8\Delta (\log n)^2$ by Lemma \ref{L:isop}. We then have
\be\label{eq:PhiG'}
|E_{G'}(S,S^c)| = |E_{G''}(\tilde{S}, \tilde{S}^c)| \geq \frac{1}{8 \Delta (\log n)^2} \min \{ |\tilde{S}|, |\tilde{S}^c| \},
\ee
where $E_{G'}(S, S^c)$ denotes the set of edges in $G'$ connecting $S$ and $S^c$. To lower bound the bottleneck ratio $\Phi_{(G', w')}$ for the weighted graph $(G', w')$, note that
\begin{equation*}
	\sum_{e\in E_{G'}(S, S^c)} w_e  \geq \frac{1}{A}\, |E_{G'}(S,S^c)|,
\end{equation*}
since all edges in $G''$, and hence in $G'$, have weights in $[1/A, A]$. On the other hand, for $S'= S$ or $S^c$, because vertices in $G'$
have degree at most $\Delta \log n$, we have
$$
\sum_{x\in S', y\in V(G')} w_{\{x, y\}} \leq A \Delta \log n \cdot |S'| \leq A \Delta \log n \cdot |\tilde S'| \,.
$$
Together with \eqref{eq:PhiG'}, this implies that, for all $S\subset V(G')$,
$$
\sum_{e\in E_{G'}(S, S^c)} w_e  \geq \frac{1}{8A^2 \Delta^2 (\log n)^3} \min \Big\{ \sum_{x\in S, y\in V(G')} w_{\{x, y\}}, \sum_{x\in S^c, y\in V(G')} w_{\{x, y\}} \Big\}.
$$
By \eqref{eq:PhiG2}, this implies $\Phi_{(G', w')} \geq \frac{1}{16 A^2 \Delta^2 (\log n)^3}$.
%, which concludes the proof of Prop.~\ref{P:Cheeger}.
\end{proof}

\section{Proof of Theorem \ref{T:main} for Expanders} \label{S:pf}

We follow the same notation as in Section \ref{S:conditioning}, where, given the edge weights $w=(w_e)_{e\in E}$, an edge $e$ is open if $w_e\in [1/A, A]$ and closed otherwise. The set of closed edges is denoted by $K$. Let $A$ be chosen as in Proposition~\ref{P:Cheeger} such that the events $(B_i)_{1\leq i\leq 4}$ defined in \eqref{eq:B1234} hold jointly with probability at least $1-C/n^\eta$ for some $\eta>0$ that can be chosen arbitrarily.
From now on we will assume $w\in \bigcap_{i=1}^4 B_i$, provided $\gamma$ in Theorem \ref{T:main} is chosen to satisfy $\gamma <\eta$.

Given $w$, and for any realisation of $\cT(K)$ under $\bP^w$, let $(G', w')$ be defined as in \eqref{eq:T'}. Let $V(G')$ and $E(G')$ denote the
vertex and edge set of $G'$. By Proposition~\ref{P:Cheeger}, we have \begin{equation} \label{eq:cheeg_final_thm}
\Phi_{(G', w')} \geq \frac{1}{16 A^2 \Delta^2 (\log n)^3}
\end{equation}
uniformly in $w\in \bigcap_{i=1}^4 B_i$ and for $\bP^w$-a.e.\ $\cT(K)$.

We will now apply Theorem \ref{T:Nach} to $(G', w')$.
Let us define the quantities $D,\alpha,\theta$ involved in conditions \eqref{eq:balanced}, \eqref{eq:mixing} and \eqref{eq:escaping}:
\begin{itemize}[leftmargin=*]
\item Condition \eqref{eq:balanced}:
since the vertices in $G'$ have degree at most $\Delta \log n$ as noted in the proof of Proposition~\ref{P:Cheeger}, and the edge weights all lie in
$[1/A, A]$, it is clear that we can take $D :=\Delta A^2 \log n$.
% in condition \eqref{eq:balanced}.
\item Condition \eqref{eq:mixing}: if $\pi$ denotes the stationary distribution of the lazy random walk on $(G', w')$, then 
\begin{equation}
    \pi_{\rm min} =\min_{v\in V(G')} \pi(v) \geq \frac{\frac{1}{A}}{2 A |E(G')|} \geq \frac{1}{A^2 \Delta n}.
\end{equation}
Then Theorem \ref{T:Cheeger} and \eqref{eq:cheeg_final_thm} together imply
\begin{equation*}
	\tmix(G', w') \leq C(A, \Delta) (\log n)^{7}
\end{equation*}
for some constant $C(A, \Delta)$ depending only on $A$ and $\Delta$. Therefore, we may choose any $\alpha\in (0,1/2)$.
% in condition \eqref{eq:mixing}.
\item Condition \eqref{eq:escaping}: to identify $\theta$ we use the trivial bound
\begin{equation}\label{eq:tmixG'}
	\sum_{t=0}^{\tmix(G', w')} (t+1) \sup_{v \in V(G')}q_t (v,v) \leq (\tmix(G', w')+1)^2 \leq (C(A, \Delta)+1)^2 (\log n)^{14} =:\theta\,.
\end{equation}
\end{itemize}
Denote now $m:=|V(G')|$, which, thanks to the event $B_4$, satisfies $n/\log n \leq m\leq n$. Since the above choices of $D, \alpha, \theta$ are uniform for $\bP^w$-a.e.\ $\cT(K)$, Theorem \ref{T:Nach} implies that  there exist $\gamma' \in (0, \eta)$, $c_1 = c_1(A,\Delta) = c_1(\mu, b, \Delta)$ and $c_2 > 0$ such that for all $n \geq 2$ and any $\eps>n^{-\gamma'}$,
\begin{equation} \label{eq:contracted_diameter}
\bP^w \Big( \frac{\sqrt{m}}{c_1(\epsilon^{-1} \log n)^{c_2}} \leq \mathrm{diam}(\cT_{(G', w')}) \leq \sqrt{m} c_1(\epsilon^{-1} \log n)^{c_2} \Big) \geq 1 - \epsilon\,.
\end{equation}
This bound is uniform for $w\in \bigcap_{i=1}^4 B_i$, and the complement of this event has probability at most $n^{-\eta}\ll \eps$. To conclude the proof of Theorem \ref{T:main}, it only remains to translate the bounds on $\mathrm{diam}(\cT_{(G', w')})$ into bounds on $\mathrm{diam}(\cT_{(G, w)})$.

By the spatial Markov property in Lemma \ref{L:SMarkov}, we can couple $\cT=\cT_{(G, w)}$ and $\cT_{(G', w')}$ such that $\cT=\cT_{(G', w')}\cup \cT(K)$. To recover $\cT$ from $\cT_{(G', w')}$, we need to undo the contraction of edges in $\cT(K)\subset K$, which consists of disjoint trees
of size at most $\log n$. Since the contraction of these trees in $\cT$ into single vertices decreases the length of paths in $\cT$, we have $\mathrm{diam}(\cT_{(G', w')}) \leq \mathrm{diam}(\cT)$. In the other direction, when we take a path in $\cT_{(G', w')}$ and undo the contraction, the worst case is when each vertex along the path is replaced by a path of length $\log n$, so $\mathrm{diam}(\cT) \leq \log n \cdot \mathrm{diam}(\cT_{(G', w')})$. These bounds and the fact that $n/\log n \leq m\leq n$ readily imply Theorem \ref{T:main}, where we can pick $0 < \gamma \leq \gamma'$ and enlarge $c_2$ to absorb the extra log factor.

\begin{remark}\label{R:universalexp}
Notice that in \eqref{eq:contracted_diameter} the exponent of the $\log n$ corrections is independent of $b$, $\Delta$ and $\mu$, % cf.~Remark \ref{R:universalconst}, 
and $\gamma$ can be taken independent of all parameters. In fact, we only impose $\gamma<\eta$, for an arbitrary large $\eta$, and $\gamma$ less than a small constant times $\alpha$ (see Appendix \ref{S:app}), where $\alpha$ can be chosen to be anything below $1/2$.
\end{remark}

\section{Finite Boxes in \texorpdfstring{$\Z^d$}{Zd}} \label{S:Zd}

Consider the lattice $\Z^d$ with edge set $E^d:= \{ \{x,y\}: x, y\in \Z^d, \|x - y\|_1 = 1\}$. For integers $L \geq 1$, consider the induced graph $G_n$ with $n = (2L+1)^d$ many vertices, defined by taking the vertex set $V_n = \Z^d \cap [-L, L]^d$. For notational sake, we shall sometimes drop the subscript $n$. To bound the diameter of the UST on $G_n$ with random edge weights, we will follow the same strategy as for expander graphs in Sections
\ref{S:expbound} and \ref{S:pf}. The key difference is that for finite boxes in $\Z^d$ with $n$ vertices, it is known that the mixing time is of order $n^{2/d}$ (see e.g.\ \cite[Theorem 1.1]{BM03}). Therefore, when we bound the parameter
$\theta$ in \eqref{eq:escaping} for the graph $(G', w')$ (recall from Section \ref{S:conditioning}), we can no longer apply the crude bound
\begin{equation*}
\sum_{t=0}^{\tmix} (t+1) \sup_{v \in V}q_t(v,v) \leq (\tmix + 1)^2
\end{equation*}
as we did in \eqref{eq:tmixG'} for expander graphs. Instead, we need to apply sharper bounds on the lazy random walk transition kernel $q_t(\cdot,\cdot)$. This will be achieved by replacing the notion of bottleneck ratio $\Phi_{(G, w)}$ in \eqref{eq:PhiG} with the notion of bottleneck profile $\Phi_{(G, w)}(r)$ in \eqref{eq:BottleProfile} below.

\subsection{Heat kernel estimates}

%To prove heat kernel estimates on $G'$ we will need a slight generalisation of the Cheeger constant defined in \eqref{eq:PhiG}.
For a finite connected weighted graph $(G, w)$, let $\pi$ denote the stationary distribution of the lazy random walk on $(G, w)$ defined as below of \eqref{eq:lazy}, and let $\pi_{\min} := \min_{v \in V} \pi(v)$. Recall from \eqref{eq:PhiGS2} that for non-empty $S\subset V$,
$\Phi_{(G, w)}(S)= \frac{1}{\pi(S)} \sum_{x\in S, y\in S^c} \pi(x) q(x,y)$. We then define the {\em bottleneck profile} by
\begin{equation} \label{eq:BottleProfile}
\Phi_{(G, w)}(r) :=
\left\{
\begin{aligned}
	\min\limits_{0 < \pi(S) \leq r} \Phi_{(G, w)}(S) & \qquad \mbox{if } \pi_{\rm min} \leq r \leq 1/2, \\
	\Phi_{(G, w)}(1/2) \qquad & \qquad \mbox{if } r> 1/2.
\end{aligned}
\right.
\end{equation}
Note that $\Phi_{(G, w)}(r)$ is decreasing in $r$, and $\Phi_{(G, w)}(1/2) = \Phi_{(G, w)}$, the bottleneck ratio defined in \eqref{eq:PhiG}. We will need the following result from \cite{MP05},
which can be thought of as a strengthening of the upper bound on the mixing time in Theorem \ref{T:Cheeger}.
\begin{theorem}[Theorem 1 in \cite{MP05}] \label{T:HeatCheeger}
Let $u,v\in V(G)$, and let $(q_t(u,v))_{t\geq 0}$ be the transition kernel of the lazy random walk on $(G,w)$. If
\begin{equation*}
	t \geq 1 + \int_{4 \min\{\pi(u), \pi(v)\}}^{\frac{4}{\xi}} \frac{4}{r \Phi_{(G, w)}^2(r)} dr,
\end{equation*}
then
\begin{equation*}
	\Big| \frac{q_t(u,v)}{\pi(v)} -1 \Big| \leq \xi\,.
\end{equation*}
\end{theorem}
We note that the above result also gives a bound on the mixing time. For our purposes, we will only use it to bound $q_t(v, v)$ by choosing an appropriate $\xi = \xi(t)$.

Let $G=(V, E)$ be the box $[-L, L]^d \cap \Z^d$ with $n$ vertices, and let $w=(w_e)_{e\in E}$ be i.i.d.\ edge weights as in Theorem \ref{T:main}.
Let $(G', w')$ be defined as in \eqref{eq:T'}, which is obtained from $(G, w)$ by conditioning on $\cT(K)$, the UST on $(G, w)$ restricted to edges $e\in E$ with $w_e\notin [1/A,A]$, which are called closed edges in a percolation process with parameter $p = p(A) = \mu([1/A,A])$. To prove Theorem \ref{T:main} for finite boxes on $\Z^d$ with $d\geq 5$, the main technical ingredient is the following bound on the bottleneck profile for $(G', w')$.

\begin{lemma} \label{L:PhiLattice}
Let $\eta > 0$. There exist constants $C = C(d,\eta) > 0$ and $p^* = p^*(d, \eta) < 1$ such that if $p(A) > p^*$, then there is an event $\tilde{\mathcal{B}}$ (see \eqref{eq:tildeB1234} below)
with $\P(w\in \tilde{\mathcal{B}}) \geq 1 - n^{-\eta}$ such that if $w \in \tilde{\mathcal{B}}$, then for $\bP^w$-a.e.\ $\cT(K)$ and any $r \geq \pi_{\rm min}$, we have
\begin{equation*}
	\Phi_{(G', w')}(r) \geq  \frac{C}{A^4 (\log n)^{d+4}} \Big(\frac{\pi_{\rm min}}{r}\Big)^{1/d}.
\end{equation*}
\end{lemma}
We will postpone the proof of Lemma \ref{L:PhiLattice} to Section \ref{S:bnprofile}.

\begin{proof}[Proof of Theorem \ref{T:main} for finite boxes in $\Z^d$]
Let $G=(V, E)$ be the box $[-L, L]^d \cap \Z^d$ with $n$ vertices, and let $w=(w_e)_{e\in E}$ be i.i.d.\ edge weights with common distribution $\mu$.
We proceed as in the proof of Theorem \ref{T:main} for expander graphs in Section \ref{S:pf}. As in Section \ref{S:conditioning}, we couple the edge weights $w=(w_e)_{e\in E}$ to a percolation process with parameter $p = p(A) = \mu([1/A,A])$, and let
$(G', w')$ be defined from $(G, w)$ as in \eqref{eq:T'}. We will again verify the three conditions of Theorem \ref{T:Nach} for $(G',w')$. In what follows, $\eta>0$ can be chosen arbitrarily, and $C$ denotes a generic constant depending only on $d$ and $\eta$, whose precise value may change from line to line.

\begin{itemize}[leftmargin=*]
	\item Condition \eqref{eq:balanced}:
	The maximum degree in $G$ is $2d$, and hence by Lemma \ref{L:scluster}, we can choose $A$ large enough such that with probability at least $1-n^{-\eta}$, every cluster of closed edges is of size at most $\log n$. Following the notation in \eqref{eq:B1234}, we denote the set of such edge weights configurations by $\tilde B_4$ (see also \eqref{eq:tildeB1234}). Then for $w\in \tilde B_4$, $(G', w')$ satisfies condition \eqref{eq:balanced} with $D = 2 d A^2 \log n$.
	
	\item Condition \eqref{eq:mixing}: Let $w \in \tilde B_4 \cap \tilde{\mathcal B}$ for the event $\tilde{\mathcal B}$ in Lemma \ref{L:PhiLattice}. The stationary distribution of the lazy random walk on $(G', w')$ then satisfies
	\begin{equation}\label{eq:piminmax}
		\frac{1}{nD} = \frac{1}{2d A^2 n \log n} \leq \pi_{\rm min} \leq \pi_{\rm max} \leq \frac{D}{n} = \frac{2d A^2 \log n}{n}.
	\end{equation}
	Using Lemma \ref{L:PhiLattice} with $r=1/2$ and Theorem \ref{T:Cheeger} gives
	\begin{equation}\label{eq:tmixG'w'}
		\begin{aligned}
			\tmix(G',w')
			\leq C A^8(\log n)^{2d+8} \pi_{\rm min}^{-2/d} \log (\pi_{\rm min}^{-1})
			\leq C A^{8 + 4/d + 1} (\log n)^{2d+8 + 2/d +1} n^{2/d}\,.
		\end{aligned}
	\end{equation}
	Since $d \geq 5$, condition \eqref{eq:mixing} is satisfied for some $\alpha > 0$.
	\item Condition \eqref{eq:escaping}: Assume again that $w\in \tilde B_4 \cap \tilde{\mathcal B}$, and let
	$$
	\xi = \xi(t) =  C_1 A^{4d} (\log n)^{d(d+4)} \pi_{\rm min}^{-1} (t-1)^{-d/2}
	$$
	for some large $C_1$ to be determined later. For small $t \geq 2$ that satisfies $\pi(v) \geq \xi(t)^{-1}$, we use the trivial bound $q_t(v,v) \leq 1 \leq (\xi(t) + 1) \pi(v)$. To bound $q_t(v, v)$ for larger $t$ with $\pi(v) < \xi(t)^{-1}$, we use the bound from Lemma \ref{L:PhiLattice}
	\begin{align*}
		1 + \int_{4\pi(v)}^{\frac{4}{\xi}} \frac{4}{r \Phi_{(G', w')}^2(r)} dr
		& \leq 1 + C A^{8} (\log n)^{2d+8} \pi_{\rm \min}^{-2/d} \int_{4\pi(v)}^{\frac{4}{\xi}} r^{2/d -1} dr \\
		& \leq 1 + C A^{8} (\log n)^{2d+8} \pi_{\rm \min}^{-2/d} \xi^{-\frac{2}{d}} \\
		&= 1 + C\cdot C_1^{-2/d} (t-1) \leq t\,,
	\end{align*}
	where we choose $C_1$ large enough so that the last inequality is satisfied. Then by Theorem \ref{T:HeatCheeger} for $t \geq 2$,
	\begin{align*}
		q_t(v,v)
		&\leq  (\xi + 1) \pi(v)
		\leq C_1 A^{4d} (\log n)^{d(d+4)} (t-1)^{-d/2} \frac{\pi_{\rm max}}{\pi_{\rm min}} + \pi(v) \\
		&\leq C A^{4d + 2} (\log n)^{d(d+4) +1} t^{-d/2} + \pi(v)\,.
	\end{align*}
	Using the above, we may bound
	\begin{align*}
		\sum_{t=0}^{\tmix(G', w')} (t+1) \sup_{v \in V(G')}q_t (v,v)
		&\leq C A^{4d + 2} (\log n)^{d(d+4) +1} \sum_{t=0}^\infty (t+1) t^{-d/2} + (\tmix+1)^2 \pi_{\rm \max} \\
		& \leq C A^{4d + 2} (\log n)^{d(d+4) +1} + C A^{20+8/d}(\log n)^{4d+4/d+19} n^{4/d - 1},
	\end{align*}
	where we used \eqref{eq:piminmax} and \eqref{eq:tmixG'w'}. Therefore, for $w\in \tilde B_4 \cap \tilde{\mathcal B}$, the escaping condition \eqref{eq:escaping} is % satisfied with $\theta= (\log n)^C$ for some $C$ that depends only on $d$, $\mu$ and $\eta$.
	satisfied with $\theta= c_1 (\log n)^{c_2}$ for some $c_1 = c_1(A,d)$ depending on $d$ and $\mu$, and some $c_2 = c_2(d)$ that depends only on $d$.
\end{itemize}

We may thus apply Theorem \ref{T:Nach} to $(G', w')$ to obtain bounds on the diameter of the UST on $(G', w')$. Deducing bounds on the diameter of the UST on $(G, w)$ then follows exactly as in Section \ref{S:pf} for expander graphs, where we again possibly need to enlarge $c_1$ and $c_2$.
\end{proof}

\begin{remark}\label{R:universalbox}
In the above proof, the exponent of the $\log n$ correction term depends on $d$ but not on the edge weight distribution $\mu$. Furthermore, for the same reasons as in Remark \ref{R:universalexp}, $\gamma$ can be taken as a universal constant as $\alpha$ can be chosen as anything below $1/10 \leq 1/2 -2/d$ for all $d \geq 5$.
\end{remark}

\subsection{Proof of Lemma \ref{L:PhiLattice}} \label{S:bnprofile}
Similar to the proof of Proposition \ref{P:Cheeger} for expander graphs, we will show that the bottleneck profile of $(G', w')$ differs from the bottleneck profile of the largest percolation cluster $\mathcal C_1$ only by a polylogarithmic factor. To this end, we define the following analogues of
the events $B_1, B_2, B_3$ and $B_4$ in \eqref{eq:B1234} for finite boxes in $\Z^d$, where $c_0$ is given in Lemma \ref{L:IsoSuperCrit} below:
\begin{equation}\label{eq:tildeB1234}
\begin{aligned}
	\tilde{B}_1 &:= \{|V(\mathcal C_1)| \geq 3n/4\}, \\
	\tilde{B}_2 &:= \{ \forall S \subseteq \mathcal C_1 \;\text{with}\, (\log n)^{\frac{d^2}{d-1}} \leq |S| \leq \frac{|V(\mathcal C_1)|}{2}, |E_{\mathcal C_1}(S, S^c)| \geq c_0 |S|^{\frac{d-1}{d}} \}, \\
	\tilde{B}_3 &:= \{ \text{all connected components of } G \backslash V(\mathcal C_1) \text{ have size at most} \, (\log n)^{\frac{d}{d-1}}\}, \\
	\tilde{B}_4 & := \{ \text{the closed edges in $K$ form clusters of size at most} \, \log n\},
\end{aligned}
\end{equation}
where ${\mathcal C}_1 = {\mathcal C}_1(p)$ denotes the largest percolation cluster in $G$ for the percolation process coupled to $w$ such that $e$ is open if $w_e\in [1/A, A]$, and $E_{\mathcal C_1}(S, S^c)$ denotes the set of edges between $S\subset V(\mathcal C_1)$ and $S^c= V(\mathcal C_1)\backslash S$ in the graph $\mathcal C_1$.  We denote by $p = p(A) = \mu([1/A, A])$ the percolation parameter.

As noted before, by Lemma \ref{L:scluster}, for any $\eta>0$ we can choose $A$ large enough such that $\P(w\in \tilde B_4) \geq 1- C n^{-\eta}$.
The following lemma gives a similar bound for $\P(w\in \tilde B_1)$.

\begin{lemma} \label{L:SizeSuperCrit}
There exist $p^*_1=p^*_1(d) \in (0, 1)$ and constants $c, C> 0$ such that for all $p \in (p^*_1, 1]$ and $n\in \N$,
\begin{equation}\label{eq:largeC1}
	\P\Big( |V(\mathcal C_1(p))| < \frac{3n}{4}\Big) \leq C \e^{- c n^{1/d}}
\end{equation}
\end{lemma}
\begin{proof}
This follows from Theorem 1.2 of \cite{DP96} by choosing $\eps = 1/4$.
\end{proof}

The next lemma gives the desired bound for $\P(w\in \tilde B_3)$.

\begin{lemma} \label{L:OutsideComponent}
For any $\eta>0$, there exists $p^*_2=p^*_2(d, \eta) \in (0, 1)$ and a constant $C > 0$, such that for all $p \in (p^*_2, 1]$ and $n\in\N$
\begin{equation}\label{eq:smallcomponents}
	\P\big( \exists \text{ a connected component } H\subseteq  G\backslash V(\mathcal C_1(p)) \text{ with } |H| \geq (\log n)^{\frac{d}{d-1}}\big) \leq  \frac{C}{n^{\eta}}.
\end{equation}
%    \begin{equation*}
	%        \P( G\backslash V(C_1(n,p)) \text{ contains a connected component } H \text{ with } |H| \geq c_2 (1 + \eta) (\log n)^{\frac{d}{d-1}}) \leq  \frac{C}{n^{\eta}}.
	%    \end{equation*}
	\end{lemma}

\begin{proof}
We follow the proof of \cite[Lemma 3.2]{PR12}, which treats a similar event on the torus $\mathbb T^d$.
By Lemma \ref{L:SizeSuperCrit}, we may first restrict to the event $\{|\mathcal C_1(p)| \geq \frac{3n}{4}\}$. We will then bound the probability in \eqref{eq:smallcomponents} by a union bound over all vertex sets $S\subset V$ in the box $G=(V, E)$ that can arise as a connected component of $G\backslash V(\mathcal C_1)$ with $(\log n)^{\frac{d}{d-1}}\leq |S|\leq \frac{n}{4}$. Equivalently, we can take the union bound over all realisations of the edge boundary set $\partial_E S\subset E$ (those edges with exactly one endpoint in $S$), where each edge in $\partial_E S$ must be closed in the percolation configuration for $S$ to be a connected component of $G\backslash V(\mathcal C_1)$. By the edge isoperimetric inequality for finite boxes in \cite[Theorem 3]{BL91}, for $|S|$ with $|S|\leq n/2$, we have
\begin{equation}\label{eq:eiso}
	|\partial_E S| \geq \min\{ |S|^{1-\frac{1}{r}} r n^{\frac{1}{r}-\frac{1}{d}}: r=1, 2, \ldots, d\}   \geq |S|^{\frac{d-1}{d}}.
\end{equation}
Since we will only consider $S$ with  $|S|\geq (\log n)^{\frac{d}{d-1}}$, this imposes the constraint that
\begin{equation}\label{eq:eiso2}
	|\partial_E S| \geq  \log n.
\end{equation}

We define a new graph $\mathcal{E}^*$ as follows. The vertices of $\mathcal{E}^*$ correspond to the edges of the box $[ -L, L]^d \cap \Z^d$, and for $e \in V(\mathcal{E}^*)$ we denote by $m_e$ the midpoint of $e$ viewed as a vector in $\R^d$. Two vertices $e,f$ of $\mathcal{E^*}$ are connected if their midpoints, $m_e$ and $m_f$, are of $\ell_\infty$ distance at most $1$. See also the construction in \cite[Section 2]{DP96}. One may verify that $\mathcal{E}^*$ has at most $d n$ many vertices and that the maximum degree is bounded by some constant $c_d$ depending only on $d$. We say that an edge set $F \subset E(G)$ is $\ast$-connected if the corresponding vertices in $\mathcal{E}^*$ are connected.

For $S\subset V$ to arise as a connected component of $G\backslash V(\mathcal C_1)$, we note that $\partial_E S$ must be a $\ast$-connected set of edges (cf.\ \cite[Lemma 2.1]{DP96}). As cited in the proof of Lemma \ref{L:scluster}, the number of $\ast$-connected $\partial_E S\subset E$ with $|\partial_E S|=k$ is then bounded by $\frac{d n}{k} (c_d \e)^k$. We can now first apply Lemma \ref{L:SizeSuperCrit} and then take a union bound over all $\ast$-connected edge sets $\partial_E S \subset E$ with  $|\partial_E S| \geq \log n$ to bound the probability in \eqref{eq:smallcomponents} by
\begin{equation*}
	C \e^{- c n^{1/d}} +  \sum_{k=\log n}^{d n} \frac{d n}{k} (c_d \e)^k   (1-p)^k \leq \frac{C}{n^\eta},
\end{equation*}
where the last bound holds uniformly in $p> p_2^*$ if $p_2^*$ is sufficiently close to $1$.
\end{proof}

Lastly, the following lemma gives the desired bound for $\P(w\in \tilde B_2)$.

\begin{lemma} \label{L:IsoSuperCrit}
There exists $c_0\in (0,1)$ such that for any $\eta>0$, there exist $p_3^* = p_3^*(d, \eta) \in (0, 1)$ and $C>0$, we have for all $p \in (p_3^*, 1]$ and $n\in \N$
\begin{equation}\label{eq:tB2bd}
	\P \Big( \exists S \subseteq V(\mathcal C_1)\,:
	(\log n)^{\frac{d^2}{d-1}} \leq |S|
	\leq \frac{|V(\mathcal C_1)|}{2},
	| E_{\mathcal C_1}(S,S^c)| \leq c_0 |S|^{\frac{d-1}{d}} \Big) \leq \frac{C}{n^{\eta}}.
\end{equation}
\end{lemma}
\begin{proof}[Proof sketch.] We can follow the arguments in \cite{BM03}. First we can prove a variant of \eqref{eq:tB2bd} where $S$ is further required
to satisfy the condition that both $S$ and $S^c := V(\mathcal C_1) \backslash S$ are connected in $\mathcal C_1$, that is, for $p$ close enough to $1$,
\begin{align}
	\P \Big( \exists \text{connected }S \subseteq V(\mathcal C_1)\, & :
	V(\mathcal C_1) \backslash S  \text{ is connected in } \mathcal C_1, \label{eq:B2DoubleConnected} \\
	& (\log n)^{\frac{d}{d-1}} \leq |S|
	\leq \frac{|V(\mathcal C_1)|}{2},
	| E_{\mathcal C_1}(S,S^c)| \leq \frac{1}{2} |S|^{\frac{d-1}{d}} \Big) \leq \frac{C}{n^{\eta}}. \notag
\end{align}
This is essentially \cite[Theorem 2.4]{BM03} (or \cite[Corollary 1.4]{Pet08}) with a quantitative probability bound. Following the proof in \cite[Section 2.4]{BM03}, the basic observation is that for any subset $S$ of the box $G=(V, E)$ with $|S| \geq (\log n)^{\frac{d}{d-1}}$, if $S$ turns out to be a connected subset of $\mathcal C_1$ such that $S^c$ is also connected in $\mathcal C_1$, then $S^c$ must lie in one of the connected components of $V\backslash S$ in $G$ (denoted by $A_r$ in \cite[Section 2.4]{BM03} with $A_r \supseteq \mathcal{C}_1 \backslash S$ and $\partial_E S = E_G(S, A_r)$). By the assumption $|S|\leq \frac{1}{2} |V(\mathcal C_1)|$, we have $|S|\leq |A_r|$, and hence by the isoperimetric inequality \eqref{eq:eiso}, the number of edges between $S$ and $A_r$ in $G$ is bounded by
$$
|E_G(S, A_r)| \geq |S|^{\frac{d-1}{d}} \geq \log n.
$$
Since we are assuming $S\subseteq V(\mathcal C_1)$, for the event $\{| E_{\mathcal C_1}(S,S^c)| \leq \frac{1}{2} |S|^{\frac{d-1}{d}}\}$ to occur, at least half of the edges in $E_G(S, A_r)$ must be closed in the percolation configuration, the probability of which can be bounded by $\exp (-c |E_G(S, A_r)|)$ for an arbitrarily large $c$ if $p$ is chosen close enough to $1$. %Could add a sentence like: union bounding over all sets of size 1/2 |E_G(S, A_r)| and then p small enough...

Fix two vertices $u \in S$ and $v \in A_r$, then as $S$ and $A_r$ are connected, $E_G(S, A_r)$ is a minimal cutset separting $u$ from $v$. That is, any path connecting $u$ and $v$ uses at least one edge in $E_G(S, A_r)$ and removing any edge from $E_G(S,A_r)$ breaks this property. By \cite[Lemma 2.9]{BM03} there is some universal $c'=c'(d)$ such that there are at most $\exp( c' |E_G(S,A_r)|$) many minimal cutsets separating $u$ from $v$. A union bound over all $u,v \in V$ and all possible choices of $E_G(S, A_r)$ with $|E_G(S, A_r)| \geq \log n$ then gives the desired probability bound of $C/n^\eta$ if $p$ is chosen close enough to $1$. For the remainder of the proof, we assume $p$ is so large such that \eqref{eq:B2DoubleConnected} holds.

To extend the bound \eqref{eq:B2DoubleConnected} to every $S$, following the proof
of Lemma 2.6 in \cite{BM03}, the key observation (cf.\ \cite[Lemma 2.5]{BM03} and \cite[Lemma 4.36]{AF14}) is that there exists a constant $c_d$ depending only on $d$ such that for any $x\in (0, 1/2]$, there exists a connected set $A\subset V(\mathcal C_1)$ with $|A|\leq x |V(\mathcal C_1)|$
such that
\begin{equation} \label{eq:B2SingleConnected}
	\tilde \varphi(x):= \inf_{S\subset V(\mathcal C_1), |S|\leq x |V(\mathcal C_1)|} \frac{|E_{\mathcal C_1}(S,S^c)|}{|S|} \geq c_d \frac{|E_{\mathcal C_1}(A,A^c)|}{|A|} =: c_d \tilde\varphi_A,
\end{equation}
i.e.\ the infimum (up to constants) of $\tilde{\varphi(x)}$ is obtained by connected sets. Note that our definition of $\tilde \varphi$ differs slightly from the definition of $\varphi$ in \cite{BM03}, although they are within constant multiples of each other, which explains the inequality and the constant $c_d$ in \eqref{eq:B2SingleConnected}.
Furthermore, if $x=1/2$, we can choose $A$ in \eqref{eq:B2SingleConnected} such that $A$ and $V(\mathcal C_1) \backslash A$ are both connected in $\mathcal C_1$. This last fact together with \eqref{eq:B2DoubleConnected} (assuming $|V(\mathcal C_1)|\geq 3n/4$, which we may thanks to Lemma \ref{L:SizeSuperCrit}),
and the crude bound $\tilde \varphi_A \geq \frac{1}{|A|}$ for $|A|\leq (\log n)^{\frac{d}{d-1}}$ imply that
$$
\P(\tilde \varphi(1/2) \geq c n^{-\frac{1}{d}} ) \geq 1-\frac{C}{n^\eta}
$$
for some $c>0$ depending on $d$. On the event $\{\tilde \varphi(1/2) \geq c n^{-1/d}\}$, we note that for all $S$ with $q |V(\mathcal C_1)| \leq |S| \leq \frac{1}{2}|V(\mathcal C_1)|$ for some fixed $0< q \leq \frac{1}{2}$ to be chosen later, we have
\begin{equation}\label{eq:c01}
	\frac{|E_{\mathcal C_1}(S,S^c)|}{|S|} \geq \frac{c}{n^{1/d}} \quad \Rightarrow \quad |E_{\mathcal C_1}(S,S^c)| \geq c' |S|^{\frac{d-1}{d}},
\end{equation}
where $c'$ depends on $d$ and $q$. This implies that \eqref{eq:tB2bd} holds with $c_0:=c'$
if we restrict to $S\subset V(\mathcal C_1)$ with $q |V(\mathcal C_1)| \leq |S|\leq |V(\mathcal C_1)|/2$.

It remains to show that \eqref{eq:tB2bd} still holds if we restrict to $S\subset V(\mathcal C_1)$ with $(\log n)^{\frac{d^2}{d-1}}\leq |S|\leq q |V(\mathcal C_1)|$. By the argument around \eqref{eq:B2SingleConnected}, we can find a connected $A\subset V(\mathcal C_1)$ with $|A|\leq |S|$ such that $\tilde \varphi_S \geq c_d \tilde \varphi_A$, which implies that 
\begin{equation}\label{eq:ECSA}
	\frac{|E_{\mathcal C_1}(S,S^c)|}{|S|^{\frac{d-1}{d}}} =\tilde \varphi_S |S|^{1/d} \geq c_d \tilde \varphi_A |A|^{1/d} = c_d \frac{|E_{\mathcal C_1}(A, A^c)|}{|A|^{\frac{d-1}{d}}}.
\end{equation}
We now consider the following two cases (to take into account case (1), which was not addressed in the proof of \cite[Lemma 2.6]{BM03}, we need to impose in \eqref{eq:tB2bd} the condition $|S|\geq (\log n)^{d^2/(d-1)}$ instead of $|S|\geq (\log n)^{\frac{d}{d-1}}$).
\begin{enumerate}
	\item If $|A| < (\log n)^{\frac{d}{d-1}}$, then the bounds $\tilde \varphi_S \geq c_d \tilde \varphi_A$ and $\tilde{\varphi_A} \geq \frac{1}{|A|}$ give
	\begin{equation}\label{eq:c02}
		\frac{|E_{\mathcal C_1}(S,S^c)|}{|S|^{\frac{d-1}{d}}} \geq  \frac{c_d}{|A|} \cdot |S|^{\frac{1}{d}} \geq  c_d.
	\end{equation}
	
	\item If $|A| \geq (\log n)^{\frac{d}{d-1}}$, then following the proof of Lemma 2.6 in \cite{BM03}, for $q\in (0,1/2)$ small enough, we can find another connected set $A'$ with $V(\mathcal C_1)\backslash A'$ also connected in $\mathcal C_1$ and $|A|\leq |A'|\leq |V(\mathcal C_1)|/2$, such that
	$$
	\frac{|E_{\mathcal C_1}(A, A^c)|}{|A|^{\frac{d-1}{d}}} \geq \frac{|E_{\mathcal C_1}(A', (A')^c)|}{|A'|^{\frac{d-1}{d}}} \geq \frac{1}{2},
	$$
	where by \eqref{eq:B2DoubleConnected}, the last inequality holds on an event with probability at least $1-C/n^\eta$. On the same event, by \eqref{eq:ECSA}, we have
	\begin{equation}\label{eq:c03}
		\frac{|E_{\mathcal C_1}(S,S^c)|}{|S|^{\frac{d-1}{d}}} \geq \frac{c_d}{2}.
	\end{equation}
\end{enumerate}
Combining \eqref{eq:c01}, \eqref{eq:c02} and \eqref{eq:c03} then gives \eqref{eq:tB2bd} with $c_0:=\min\{c', c_d/2\}$. 
\end{proof}

We are now ready to prove Lemma \ref{L:PhiLattice} along the same lines as in the proof of Lemma \ref{L:isop} and Proposition \ref{P:Cheeger} for expander graphs.

\begin{proof}[Proof of Lemma \ref{L:PhiLattice}]
Let $(\tilde B_i)_{1\leq i\leq 4}$ be defined as in \eqref{eq:tildeB1234}, and let $\tilde{\mathcal B}:=\cap_{i=1}^4 \tilde B_i$. By Lemmas \ref{L:scluster}, \ref{L:SizeSuperCrit} and \ref{L:OutsideComponent} and Lemma \ref{L:IsoSuperCrit}, for any $\eta > 0$, we can choose $A$ (and thus $p = p(A)$) large enough such that
\begin{equation*}
	\P(\tilde{\mathcal B}) \geq 1 - \frac{C}{n^\eta}
\end{equation*}
uniformly in $n$. We will assume the edge weights $w$ are in $\tilde{\mathcal{B}}$ from now on.

As in Section \ref{S:conditioning}, let $\cT(K)$ be any realisation of the uniform spanning tree $\cT=\cT_{(G, w)}$ restricted to the set of closed edges $K$, and we condition on this realisation. Recall from \eqref{eq:T'} and \eqref{eq:G''} that $G''$ is the graph obtained by removing edges in $K$ that are not in $\cT(K)$, while $G'$ is obtained by contracting edges in $G''$ that are in $\cT(K)$.

First, we will show that for any $S \subset V(G'')$ with $1 \leq |S| \leq n/2$, we have
\begin{equation} \label{eq:LatticeIsoMain}
	|E_{G''}(S,S^c)| \geq \frac{c}{(\log n)^{d+1}} |S|^{\frac{d-1}{d}},
\end{equation}
for some constant $c>0$ independent of $n$ and $\cT(K)$. As in the proof of Lemma \ref{L:isop}, let $S_1 = S \cap V(\mathcal C_1)$ and $S_2 = S \backslash S_1$, and let $L_1, \ldots, L_M$ denote the components of $G \backslash V(\mathcal C_1)$ that contain some vertex in $S_2$. Since $w\in \tilde{B}_3$, we have $|S_2| \leq M (\log n)^{d/(d-1)}$.

If $(\log n)^{\frac{d^2}{d-1}} \leq |S_1| \leq |S|\leq n/2 $, then the event $\tilde{B}_1 \cap \tilde{B}_2$ guarantees that
\begin{equation} \label{eq:latticeS1}
	|E_{\mathcal C_1}(S_1, S_1^c)|
	\geq c_0 \min\{ |S_1|^{\frac{d-1}{d}}, |V(\mathcal C_1)\backslash S_1|^{\frac{d-1}{d}}\} \geq \frac{c_0}{2} |S_1|^{\frac{d-1}{d}}\,,
\end{equation}
where we used the fact that $|V(\mathcal C_1)\backslash S_1|\geq |S_1|/2$ on the event $\tilde B_1=\{|V(\mathcal C_1)|\geq 3n/4\}$.

If $|S_1|< (\log n)^{\frac{d^2}{d-1}}$, we have the trivial bound
\begin{equation}\label{eq:EC1S2}
	|E_{\mathcal C_1}(S_1, S_1^c)| \geq 1 \geq \frac{1}{(\log n)^d} |S_1|^{\frac{d-1}{d}}.
\end{equation}
Note that compared with \eqref{eq:latticeS1}, the bound \eqref{eq:EC1S2} is uniform in $S_1=S\cap V(\mathcal C_1)$.

To bound $|E_{G''}(S,S^c)|$, we distinguish between two cases. If $|S_1| \geq \frac{M}{4d}$, then
\begin{equation*}
	\frac{|S_1|}{|S_2|} \geq \frac{\frac{M}{4d}}{M (\log n)^{\frac{d}{d-1}}} \geq \frac{1}{4d (\log n)^{\frac{d}{d-1}}} \quad \text{and hence} \quad \frac{|S_1|}{|S|} \geq \frac{1}{1+4d (\log n)^{\frac{d}{d-1}}}.
\end{equation*}
Since $E_{\mathcal C_1}(S_1, S_1^c) \subset E_{G''}(S,S^c)$, this bound with \eqref{eq:latticeS1} and \eqref{eq:EC1S2} imply \eqref{eq:LatticeIsoMain} for some $c>0$.

If instead $|S_1| < \frac{M}{4d}$, then as in the proof of Lemma \ref{L:isop}, we have $|E_{G''}(S,S^c)| \geq M/2$ and
\begin{equation*}
	\frac{|E_{G''}(S,S^c)|}{|S|^{\frac{d-1}{d}}} \geq \frac{M/2}{(|S_1| +|S_2|)^{\frac{d-1}{d}}} \geq \frac{M/2}{(\frac{M}{4d} + M (\log n)^{d/(d-1)})^{\frac{d-1}{d}}} \geq \frac{c}{\log n}.
\end{equation*}
Therefore, \eqref{eq:LatticeIsoMain} is still satisfied for some $c>0$.

\smallskip

As in the proof of Proposition \ref{P:Cheeger} for expander graphs, we now use \eqref{eq:LatticeIsoMain} to bound the bottleneck profile $\Phi_{(G', w')}(r)$ on the contracted graph $(G', w')$. Fix $r\in [\pi_{\rm min}, 1/2]$, and let $S \subset V$ be non-empty with $\pi(S) \leq r$. Note that $r/\pi_{\rm max}\leq |S| \leq r/\pi_{\rm min}$. Let $\tilde{S}$ be the pre-image of $S$ under the contraction from $G''$ to $G'$. Then by \eqref{eq:LatticeIsoMain}, we have
\begin{equation} \label{eq:LatticeExpansion}
	|E_{G'}(S, S^c)| = |E_{G''}(\tilde S, \tilde S^c)| \geq \frac{c}{(\log n)^{d+1}}   \min\{ |\tilde S|^{\frac{d-1}{d}}, |\tilde S^c|^{\frac{d-1}{d}} \}.
\end{equation}
Under the event $\tilde B_4$, each vertex in $S$ can be ``uncontracted'' to at most $\log n$ many vertices in $\tilde S$. Therefore, $|\tilde S|\leq |S| \log n $ and $G'$ has maximal degree at most $2d \log n$. Also, recall that the edge weights $w'(e)$ are in $[1/A, A]$. We then have
\begin{equation} \label{eq:SizeTildeS}
	\frac{|S^c|}{|S|} \geq \frac{(1-r) \pi_{\rm min}}{r \, \pi_{\rm max}} \qquad \text{and} \qquad  \frac{|\tilde S^c|}{|\tilde S|}  \geq \frac{1}{\log n}\cdot \frac{(1-r) \pi_{\rm min}}{r \pi_{\rm max}} \geq \frac{1}{2d A^2 (\log n)^2}.
\end{equation}
Inserting \eqref{eq:SizeTildeS} into \eqref{eq:LatticeExpansion} gives
\begin{equation*}
	|E_{G'}(S, S^c)| \geq \frac{c}{2d A^2 (\log n)^{d+3}}  |S|^{\frac{d-1}{d}}.
\end{equation*}
By the definition of $\Phi_{(G', w')}(S)$ in \eqref{eq:PhiGS}, the fact that $w'(e)\in [1/A, A]$ and $G'$ has maximal degree at most $2d \log n$, we obtain
\begin{equation*}
	\Phi_{(G', w')}(S) \geq \frac{|E_{G'}(S,S^c)| \cdot \frac{1}{A}}{2 |S| \cdot 2d A \log n}
	\geq \frac{c}{ 8 d^2 A^4 (\log n)^{d+4}} |S|^{-\frac{1}{d}}
	\geq \frac{c}{ 8 d^2 A^4 (\log n)^{d+4}} \Big(\frac{\pi_{\rm min}}{r}\Big)^{\frac{1}{d}},
\end{equation*}
which concludes the proof of Lemma \ref{L:PhiLattice}.
\end{proof}

\section{Other Graphs and Limitations} \label{S:GeneralCounter}

As noted in Remark \ref{R:seq}, for a sequence of $b$-expanders $(G_n, w_n)$ with $n$ vertices, maximal degree uniformly bounded by some $\Delta<\infty$, and i.i.d.\ edge weights $(w_n(e))_{e\in E_n}$ with common distribution $\mu$ such that $\mu(0,\infty)=1$, Theorem \ref{T:main} implies that with high probability as $n\to\infty$, the diameter of the uniform spanning tree $\cT_{(G_n, w_n)}$ is of order $\sqrt{n}$ up to polylogarithmic factors. In this section, we give an example showing that this conclusion is false if we drop the assumption of uniformly bounded degree. Additionally, we will discuss
possible extensions to general graphs, provided they ``expand'' well enough and have bounded degree.

\subsection{Unbounded Degree Counter Example}\label{S:counterex}

Let $K_n$ be the complete graph with $n$ vertices. Clearly, the isoperimetric constant $h_{K_n}$ as defined in \eqref{eq:hG} is of order $n$, so for any $b > 0$, $K_n$ is a $b$-expander for $n$ large enough. However, the maximum degree is $\Delta = n-1$, so the condition of bounded degrees in Theorem \ref{T:main}~(i) do not hold uniformly for $(K_n)_{n\in\N}$. Indeed, the conclusion in Remark \ref{R:seq} fails for certain choices of the edge weight distribution $\mu$, as shown in the following result.

\begin{proposition}
Let $K_n$ be the complete graph with $n$ vertices and assign each edge $e$ the weight $w_n(e) = \exp(\exp(U_e^{-1}))$, where $(U_e)_{e \in E(K_n)}$ are i.i.d.\ and uniformly distributed in $[0,1]$. Let $M^{(n)}$ be the (a.s.\ unique) spanning tree $T$ on $K_n$ that maximises $w_n(T):=\prod_{e \in T} w_n(e)$. Then
\begin{equation} \label{eq:convergeT1}
	\widehat \P( \cT = M^{(n)}) \xrightarrow{n \rightarrow \infty} 1.
\end{equation}
As a consequence, with high probability, the diameter of the uniform spanning tree $\cT_{(K_n, w_n)}$ is of order $n^{1/3}$ as $n\to\infty$.
\end{proposition}
\begin{proof}
Let $T_1$ and $T_2$ be the spanning trees in $K_n$ with the largest and second largest weight $w_n(T)=\prod_{e \in T} w_n(e)$, respectively.
Note that $T_1 = M^{(n)}$, the minimum spanning tree on $K_n$ with edge variables $(U_e)_{e\in E(K_n)}$, see e.g.~\cite{ABGM17}. We have the following facts:
\begin{enumerate}[label={\arabic*)}]
	\item The number of spanning trees on $K_n$ is $n^{n-2}$ (Caley's formula). \label{enu:Cayley}
	\item $T_1$ and $T_2$ differ by a single edge (otherwise we can swap an edge in $T_2$ with an edge in $T_1$ to obtain a spanning tree $T_3$ with $w_n(T_2) < w_n(T_3) < w_n(T_1)$). \label{enu:T1T2}
	\item If $(U_i)_{1\leq i \leq m}$ is a collection of i.i.d.\ uniform random variables on $[0,1]$ and $X = \min_{i \neq j} |U_i - U_j|$, then
	\begin{equation} \label{eq:uniform_gap_law}
		\P( X > t) = (1 - (m-1) t)^m \qquad \text{for } 0 \leq t \leq \frac{1}{m-1}.
	\end{equation}
        Namely, by exchangeability
        \begin{align*}
            \P( X > t) &= m! \, \P( U_1 < U_2 - t, U_2 < U_3 -t, \ldots, U_{m-1} < U_m -t) \\
            &= m! \int^1_{(m-1)t} \int^{u_m - t}_{(m-2)t} \ldots \int^{u_3 - t}_{t} \int^{u_2 -t}_{0} du_1 du_2 \ldots du_{m-1} du_{m}.
        \end{align*}
        Using induction one can verify that for $a \geq (m-1)t$
        \begin{equation*}
            \int^a_{(m-1)t} \int^{u_m - t}_{(m-2)t} \ldots \int^{u_3 - t}_{t} \int^{u_2 -t}_{0} du_1 du_2 \ldots du_{m-1} du_{m} = \frac{1}{m!}(a - (m-1)t)^m,
        \end{equation*}
        from which \eqref{eq:uniform_gap_law} follows. This implies that $\P( X \leq m^{-2} (\log m)^{-1}) \rightarrow 0$. \label{enu:minGapU}
 
	\item The tree $T_1$ may be constructed using Kruskal's algorithm (see e.g.\ the introduction of \cite{ABGM17}). Using a coupling between an Erd\H{o}s--R\'enyi random graph $G_{n,p}$ and the random variables $(U_e)_{e \in K_n}$, we have that the connected components of $G_{n,p}$ and $G_{n,p} \cap T_1$ coincide with each other. As the probability that $G_{n,p}$ is connected for $p_n=2 \log n/n$ tends to one as $n$ diverges (see e.g.\ Chapter 4 of \cite{FK16}),  with high probability every edge $e$ in $T_1$ has $U_e \leq 2 \log n/n$. \label{enu:ConnectionP} %for some reason adding a label creates some vertical space below....?
\end{enumerate}
In view of items \ref{enu:Cayley} and \ref{enu:T1T2} above and the law of the uniform spanning tree $\cT_{(K_n, w_n)}$ defined in \eqref{eq:UST},
to prove \eqref{eq:convergeT1}, it suffices to show that
\begin{equation*}
	\frac{n^{n-2} w_n(T_2)}{w_n(T_1)} \xrightarrow{n \rightarrow \infty} 0\,.
\end{equation*}
To this end, let $e \in T_1 \backslash T_2$ and $f \in T_2 \backslash T_1$ be the two edges in which $T_1$ and $T_2$ differ, and define the gap $g := U_f - U_e$. Then by item \ref{enu:minGapU} with $m=\binom{n}{2}$, with high probability
\begin{equation*}
	\frac{g}{U_e U_f} \geq n^{-4} \log^{-1} n.
\end{equation*}
Item \ref{enu:ConnectionP} gives that with high probability $\max_{e\in T_1} U_e\leq \frac{2\log n}{n}$, and hence
\begin{align*}
	\frac{n^{n-2} w_n(T_2)}{w_n(T_1)} &= n^{n-2} \exp \Big( \exp(U_f^{-1}) - \exp(U_e^{-1})\Big) \\
	&= n^{n-2} \exp \Bigg( -\exp(U_e^{-1}) \Big(1 - \exp \big(- \frac{g}{U_e U_f} \big) \Big)\Bigg) \\
	&\leq n^{n-2} \exp \Bigg( -\exp(U_e^{-1}) \Big(1 - \exp \big(- n^{-4} \log^{-1} n \big) \Big)\Bigg) \\
	&\leq n^{n-2} \exp \Big( -\exp\big(\frac{n}{2 \log n} \big ) \frac{1}{2n^4 \log n} \Big) \rightarrow 0 \, ,
\end{align*}
with high probability, where we used the bound $(1-\e^{-x}) \geq x/2$ for $0 < x \leq 1$.

It is known from \cite{ABGM17} that $n^{-1/3} M^{(n)}$ converges in distribution to a random compact metric space. Therefore, with high probability as $n\to\infty$, the diameter of $T_1$, and hence $\cT_{(K_n, w_n)}$, is of order $n^{1/3}$. %It's a bit overkill... 
\end{proof}

\begin{remark}
We believe that assigning i.i.d.\ weights $\exp(U_e^{-1/3})$ to each edge $e\in E(K_n)$ already results in a uniform spanning tree $\cT$ with diameter
of order $n^{1/3}$ for typical realisations of $(U_e)_{e\in E(K_n)}$, although the law of $\cT$ no longer concentrates on the minimum spanning tree $T_1$. On the other hand, for edge weights $\exp(U^{-\gamma}_e)$ with $0 < \gamma < 1/3$, we conjecture that the diameter of the UST is of order $n^\alpha$ for some $\alpha \in (1/3, 1/2)$ for typical realisations of $w_n$. 
\end{remark}

\subsection{Bounded Degree Graphs With Good Bottleneck Profile}

Assume that $(G_n)_{n \in \N}$ is a sequence of bounded degree graphs with good expansion properties, say for the unweighted graph $\Phi_{G_n}(r) \geq c (\pi_{\rm min}^{-1} r)^{-\frac{1}{\beta}}$, where $\beta > 4$ so that there exists an $\alpha>0$ with $2/\beta - \alpha \leq 1/2$, in line with the mixing condition \eqref{eq:mixing}. Furthermore, assume that for some $k \geq 0$ and percolation parameter $p$ arbitrarily close to $1$ (both independent of $n$), one can prove that
\begin{equation} \label{eq:ConjIso}
`` \;\Phi_{G_n}(r) \geq c (\pi_{\rm min}^{-1} r)^{-\frac{1}{\beta}} \;\Longrightarrow\; \Phi_{G_n'} (r) \geq c (\log n)^{-k} (\pi_{\rm min}^{-1} r)^{-\frac{1}{\beta}} \text{ with high probability}",
\end{equation}
where $G_n'$ is the graph obtained by conditioning the uniform spanning tree on the set of closed edges and contracting the resulting connected components, cf.~Section \ref{S:conditioning}. One may then obtain bounds on the mixing time and the transition probabilities of $G'_{n}$ as carried out in Section \ref{S:Zd}.
In other words, for bounded degree graphs, verifying the implication \eqref{eq:ConjIso} would imply that the diameter of the UST is of order $\sqrt{n}$ up to factors of $(\log n)^c$. 

However, it is not clear how the implication \eqref{eq:ConjIso} can be proven for an arbitrary sequence of bounded degree graphs with good expansion. 
In the proof of Theorem \ref{T:main}, we used additional knowledge about the structure of expander graphs and supercritical percolation clusters on $\Z^d$. But there is hope that \eqref{eq:ConjIso} can be proved for a larger family of graphs since in the coupling between the random edge weights and the bond percolation process, we can choose the percolation parameter $p$ arbitrarily close to $1$ and condition on very subcritical clusters.

\appendix 
\section{Proof Sketch for Theorem \ref{T:Nach}} \label{S:app}

We sketch here how the proof of \cite[Theorem 1.1]{MNS21} can be adapted to prove Theorem \ref{T:Nach}.

\begin{proof}[Proof sketch for Theorem \ref{T:Nach}]
	One difference between Theorem \ref{T:Nach} and \cite[Theorem 1.1]{MNS21} is that the latter only considers unweighted graphs. Furthermore, in the bounds for $\mathrm{diam}(\cT_{(G, w)})$ equivalent to \eqref{eq:Nach1}, the prefactors of $\sqrt{n}$ were only given in terms of a generic constant $C$ in \cite{MNS21}. Here we make $C$ depend explicitly on the other quantities and show that it can be taken of the form $(C'D \theta \epsilon^{-1})^{k}$. We now explain how the proof of \cite[Theorem 1.1]{MNS21} can be adapted to account for these differences.
	
	First note that the conditions \eqref{eq:balanced}, \eqref{eq:mixing} and \eqref{eq:escaping} are natural analogues of the conditions (bal), (mix), and (esc) in \cite[Theorem 1.1]{MNS21} for weighted graphs. Let $(G,w)$ be a weighted graph with $|V|=n$ that satisfies \eqref{eq:balanced}, \eqref{eq:mixing} and \eqref{eq:escaping} with constants $D, \alpha$, and $\theta$. We note that the stationary distribution $\pi$
	of the lazy random walk on $(G, w)$ still satisfies $\frac{1}{Dn}\leq \pi(v)\leq \frac{D}{n}$ for all $v\in V$, and the crucial bounds
	in (1) and (2) of \cite{MNS21} still hold. Following the notation in \cite{MNS21} (see (7) and Claim 2.5 therein), we denote
	\begin{equation*}
		r := n^{1/2 - \alpha/3},\;\; s: =n^{1/2 - 2\alpha/3},\;\; q := r/\sqrt{n} = n^{-\alpha/3}\,.
	\end{equation*}
	
	We first point out how the bounds in Section 2 of \cite{MNS21} can be quantified, which will give the lower bound on $\mathrm{diam}(\cT_{(G, w)})$ in Theorem \ref{T:Nach} for suitable choices of $k$. By tracking the precise constants in each instance of $\preceq$ and $\succeq$, it can be checked that the bounds in Claims 2.2, 2.3, 2.5 and 2.6 in \cite[Section 2]{MNS21} can be quantified as follows:
	\begin{center}
		\begin{tabular}{c|c}
			& \textbf{Precise Bounds} \\\hline
			Claim 2.2 & $\leq D(\theta + 2 D) /n $\\
			Claim 2.3 & $\geq \big(2D^2( \theta + 2D)\big)^{-1}$, assuming $6r^2D/n\leq 1/2$ \\
			Claim 2.5 & $\geq q^2/16D^9(\theta + 2D)^3 =: 2C_0 q^2$ \\
			Claim 2.6 & $\leq D^2q^4$
		\end{tabular}
	\end{center}
	where $C_0:= \big(32 D^9(\theta + 2D)^3\big)^{-1}$. The bounds above lead to more precise bounds in the proof of \cite[Claim 2.4]{MNS21} as follows:
	\begin{center}
		\begin{tabular}{c|c}
			Equation & \textbf{Precise Bounds} \\\hline
			(9) &  $\leq 2D \beta^2$ \\
			(11) & $\P(\sum_{i=1}^N {\rm Cap}_r(\cdot)\geq \frac{1}{2} C_0 \beta q) \geq 1- 2\e^{-\beta/2D^2 q^{1/2}}$, assuming $C_0^{-1}\ll q^{-1/4}$ \\
			(12) & $\P(\sum_{i\neq j} {\rm Close}_r(\cdot) \geq  \frac{1}{2} C_0 \beta q) \leq D^2 \beta^2 q^{1/2}$, assuming $(C_0\beta)^{-1}\ll q^{-1/2}$
		\end{tabular}
	\end{center}
	Claim 2.4 in \cite{MNS21} then becomes
	\begin{equation*}
		\P\big(\text{Cap}_r(LE(X) )\geq C_0 \beta q\big) \geq 1- f_1(n, \beta)
	\end{equation*}
	with
	\begin{equation*}
		f_1(n,\beta) := \underbrace{2D \beta^2}_{(9)} + \underbrace{\beta/qn^2}_{(10)} + \underbrace{2 \exp\{-\beta/2D^2q^{1/2}\}}_{(11)} + \underbrace{D^2\beta^2 q^{1/2}}_{(12)}.
	\end{equation*}
	This strengthened version of Claim 2.4 can then be applied in the proof of Claim 2.8, which now states
	\begin{equation*}
		\P\big( \text{Cap}_r(LE(X)) \geq  C_0 \beta q \quad \text{and} \quad |LE(X)| \leq \beta^{-3} \sqrt{n}\big)  \geq 1 - f_2(n, \beta),
	\end{equation*}
	with
	\begin{equation*}
		f_2(n,\beta) := \beta^3 +\underbrace{ 2 D s \beta^{-2} n^{-1/2} + 3 C_0 \beta^2}_{(*)} + \underbrace{2 C_0 \beta^2 + n^{-2} + D\beta^2+f_1(n,\beta)}_{(**)} + \beta\,,
	\end{equation*}
	where $(*)$ and $(**)$ come, respectively, from the two cases $F^\star$ and $F_\star$  in the proof of Claim 2.8 in \cite{MNS21}, and we need to assume
	$$
	C_0\beta q \geq \frac{Dq}{\sqrt{n}} \quad \Longleftrightarrow \quad 32 D^{10}(\theta + 2D)^3 \leq \beta \sqrt{n},
	$$
	which is used in the last equation display on page 277 of \cite{MNS21}.
	
	The strengthened version of Claim 2.8 can then be applied in the proof of Claim 2.9, which now states
	\begin{equation*}
		\P\Big( Y_{\tau_{LE(X)}} \neq \rho \ \ \text{and} \ \ \text{Cap}_r(LE(Y)) \geq C_0 \beta^4 q \ \ \text{and} \ \ |LE(Y)| \leq \frac{\sqrt{n}}{\beta^3}\Big) \geq 1- f_3(n, \beta),
	\end{equation*}
	with 
	\begin{equation*}
		f_3(n,\beta) := f_2(n,\beta) + \underbrace{3 \beta C_0^{-1}}_{\tau^Y_\rho < \tau^Y_W} + \underbrace{D \beta + \beta}_{(14)} +f_1(n,\beta^4) + D\beta + D\beta^8,
	\end{equation*}
	where the terms are collected in the order they appear in the proof of Claim 2.9 in \cite{MNS21}.
	
	If $D, \theta \leq n^{\gamma}$ for some $\gamma>0$ sufficiently small ($\gamma=\alpha/325$ will suffice), then for all $\eps\in (n^{-\gamma}, 1)$ and the choice $\beta:= \delta C_0\eps$ for some $\delta>0$ small and independent of $D, \theta$, $\eps$ and $n$, it can be seen that all the assumptions are satisfied and $f_1(n, \beta), f_2(n, \beta), f_3(n, \beta) \leq \eps/3$. The strengthened versions of Claims 2.8 and 2.9 can then be applied to deduce a quantitative version of \cite[Theorem 2.1]{MNS21} for weighted graphs, which now reads as: for all $\eps\in (n^{-\gamma}, 1)$, and with $\beta=\delta C_0\eps$,
	\be\label{eq:varphibd}
	\P( |\varphi| \leq 2 \beta^{-3} \sqrt{n} \text{ and Cap}_r(\varphi) \geq C_0 \beta^4 q) \geq 1 - \epsilon,
	\ee
	where $|\varphi|$ is the length of the loop erased random walk between two independently chosen vertices according to the stationary distribution $\pi$.
	As shown in \cite{MNS21}, the lower bound on ${\rm Cap}_r(\varphi)$ implies $|\varphi| \geq C_0\beta^4 \sqrt{n}/D$, from which the lower bound on $\mathrm{diam}(\cT_{(G, w)})$ in Theorem \ref{T:Nach} follows easily.
	
	We now point out how the bounds in Section 3 of \cite{MNS21} can be modified to give the upper bound on $\mathrm{diam}(\cT_{(G, w)})$ in Theorem \ref{T:Nach}.
	First note that the results on effective conductance in Section 3.1 of \cite{MNS21} still hold if we define the effective conductance between two
	disjoint set of vertices $W, S\subset V$ by
	\begin{equation*}
		\cC_{\rm eff}(W\leftrightarrow S) = 2 \Vert w\Vert  \sum_{u\in W} \pi(u) \bP_u(\tau_S <\tau_W^+),
	\end{equation*}
	where $\Vert w\Vert :=\sum_{e\in E} w_e$, $\bP_u$ is the law of the lazy random walk $X$ starting from $u$, $\tau_S$ is the first hitting time of $S$,
	and $\tau_W^+$ is the first return time to $W$. The degree of a vertex $deg(u)$ should be replaced by $w(u):=\sum_{v\in V} w_{\{u, v\}}$. The path measure
	$\mu_W$ in \cite[Section 3.2]{MNS21} should be modified using the (non-lazy) random walk on $(G, w)$ with initial distribution $\pi(u)/\pi(W)$, $u\in W$. All the claims in \cite[Section 3]{MNS21} have identified the constants explicitly, so we only need to identify the constants in the last step.
	
	Conditioned on the event  in \eqref{eq:varphibd} for the loop erased random walk path $\varphi$, we apply Theorem 3.1 in \cite{MNS21} to the set
	$W=\varphi$ with constants $A=2\beta^{-3}$ and $\chi=C_0\beta^4$, where $\beta$ and $\eps$ are chosen as in \eqref{eq:varphibd}. To have the same $\eps$
	for the probability bound in Theorem 3.1, going through its proof in \cite[page 293]{MNS21} shows that in Theorem 3.1, we can choose
	\begin{align*}
		C := C(D, \theta, \alpha, \epsilon) & =  \frac{2 C_3 A}{\epsilon} \\
		& = 82944 D^4 \mathcal{B}_W(G)^3 \log(192 D \mathcal{B}_W(G)) A \epsilon^{-1} \\
		& \leq 82944 D^4 \Big(\theta + 2D + \frac{36D}{\chi^2}\Big)^3 \log\Big(192 D (\theta + 2D + \frac{36D}{\chi^2})\Big) A \epsilon^{-1}.
	\end{align*}
	As in \cite[page 293]{MNS21}, this then leads to the upper bound on $\mathrm{diam}(\cT_{(G, w)})$,
	\begin{equation*}
		\P\big(\mathrm{diam}(\cT_{(G, w)}) \leq (A + C) \sqrt{n}\big) \geq (1-\eps)^2 \geq 1-2\eps.
	\end{equation*}
	It is easy to check that both $A$ and $C$ can be bounded by a constant multiple of $(D \theta \epsilon^{-1})^{k}$ for some $k>0$. Applying the above bound with $\eps/2$ instead of $\eps$ and choosing a large enough $k$ so that the lower bound on $\mathrm{diam}(\cT_{(G, w)})$ in Theorem \ref{T:Nach} also holds
	then gives Theorem \ref{T:Nach} in its stated form. One feasible choice of $k$ and $\gamma$ in Theorem \ref{T:Nach} is $k=400$ and $\gamma = \alpha/325$,
	although we have not attempted to improve these exponents.
\end{proof}

%%%%%%%%%%%%%%%%%%%%%%%%%%%%%%%%%%%%%%%%%%%%%%
%% Support information, if any,             %%
%% should be provided in the                %%
%% Acknowledgements section.                %%
%%%%%%%%%%%%%%%%%%%%%%%%%%%%%%%%%%%%%%%%%%%%%%
\section*{Acknowledgements}
R.~Sun is supported by NUS Tier 1 grant A-8001448-00-00. M.~Salvi acknowledges financial support from MUR Departments of Excellence Program MatMod@Tov, CUP E83C23000330006, the MUR Prin project \emph{GRAFIA} and the INdAM group GNAMPA. M.~Salvi also thanks the hospitality of the National University of Singapore where part of this project was carried out.

\bibliographystyle{plain}
\bibliography{RSTRE}

\end{document}